\RequirePackage{fix-cm}

\documentclass[smallextended,numbook,runningheads]{svjour3p}       

\usepackage{color}
\usepackage{multirow}
\usepackage{enumitem}

\usepackage{amsmath}
\usepackage{amssymb}
\usepackage{amsthm} 
\usepackage{mathptmx}
\usepackage{tikz}
\usepackage{tikz-cd}
\usepackage{placeins}
\usepackage{caption}
\usepackage{a4wide}
\usepackage{subcaption}
\usepackage{booktabs}
\usepackage{diagbox}
\usepackage{lipsum}
\usepackage[scientific-notation=true]{siunitx}
\renewcommand{\div}{\operatorname{div}}
\newcommand{\norm}[1]{\| #1 \|}
\newcommand{\norw}[2]{\| #1 \|_{#2}}
\newcommand{\triplenorm}[1]{{\left\vert\kern-0.25ex\left\vert\kern-0.25ex\left\vert #1 
    \right\vert\kern-0.25ex\right\vert\kern-0.25ex\right\vert}}
\newcommand{\Wnorm}[1]{\triplenorm{#1}}
\newcommand{\inner}[2]{\left( #1, #2\right)}
\newcommand{\grad}{\nabla}
\newcommand{\eps}{\varepsilon}
\newcommand{\trace}{\mathrm{tr}}
\newcommand{\dx}{\, \mathrm{d} x}
\newcommand{\ds}{\, \mathrm{d} s}

\newcommand{\foralls}{\ensuremath{\forall\,}}
\newcommand{\ls}{\lesssim}  
\newcommand{\triang}{\mathcal{T}}
\newcommand{\CG}{P}
\newcommand{\DG}{DG}
\newcommand{\R}{\mathbb{R}}

\newcommand{\T}{\ensuremath{\mathcal{T}}}

\usepackage{lineno}

\begin{document}

\journalname{BIT}
\title{Accurate discretization of poroelasticity without Darcy stability}
\subtitle{Stokes-Biot stability revisited}
\titlerunning{Stokes-Biot stability revisited}


\author{Kent-Andre Mardal \and Marie E. Rognes \and Travis B. Thompson$^\dagger$}
\authorrunning{K.-A.~Mardal, M.E.~Rognes and T.~Thompson}

\institute{K.-A. Mardal \at Dept.~of Mathematics, Univ. of Oslo and Simula Research Laboratory, Norway. %
\and M.E.~Rognes \at ~Dept.~of~Sci.~Comput.~and Num.~Anal.~~Simula Research Laboratory, Norway. %
\and T.~Thompson \at ~Dept.~of~Sci.~Comput.~and Num.~Anal.~~Simula Research Laboratory, Norway and Mathematical Institute, Oxford University. \\
$\dagger$  Corresponding Author \email{thompsont@maths.ox.ac.uk}}

\maketitle 

\begin{abstract}
In this manuscript we focus on the question: what is the correct notion of 
Stokes-Biot stability?  Stokes-Biot stable discretizations have 
been introduced, independently by several authors, as a means of discretizing 
Biot's equations of poroelasticity; such schemes  
retain 
their stability and convergence properties, with respect to appropriately 
defined norms, in the context of a vanishing storage coefficient and a vanishing 
hydraulic conductivity.  The basic premise of a Stokes-Biot stable discretization 
is: one part Stokes stability and one part mixed Darcy stability.  In this manuscript 
we remark on the observation that the latter condition can be generalized to a 
wider class of discrete spaces.  In particular: a parameter-uniform inf-sup 
condition for a mixed Darcy sub-problem is not strictly necessary to retain the 
practical advantages currently enjoyed by the class of Stokes-Biot stable 
Euler-Galerkin discretization schemes.  
\end{abstract}

\section{Introduction}
\label{sec:introduction}

In this note, we consider a three-field formulation of the
time-dependent Biot equations %
describing flow through an isotropic, porous and linearly elastic medium, 
reading as: find the elastic displacement $u$, the Darcy flux $z$ and the
  (negative) fluid pressure $p$ such that
\begin{subequations}
  \begin{eqnarray}
    \label{eq:biot:a}
    - \div \sigma(u) - \alpha\grad p &=& f, \\        
    \label{eq:biot:b}
    \frac{1}{\kappa} z - \grad p &=& g, \\  
    \label{eq:biot:c} 
    \alpha\div \partial_t u + \div z - c_0 \, \partial_t p &=& s, 
  \end{eqnarray}
  \label{eq:biot}%
\end{subequations}
for a given body force $f$, source $s$, and given $g$ (typically $g =
0$) over a domain $\Omega \subset \R^d$ ($d = 1, 2, 3$). The expression 
$\sigma(u)$ denotes the isotropic elastic stress tensor, 
$\sigma(u) = \mu \eps(u) + \lambda \trace \eps(u)$, where $\trace$ is the 
matrix trace. The material parameters are the elastic Lam\'{e} parameters 
$\mu$ and $\lambda$, the Biot-Willis coefficient $\alpha$, the storage 
coefficient $c_0 \geq 0$ and the hydraulic conductivity 
$\kappa = K/\mu_f > 0$, in which $K$ is the material permeability, 
and $\mu_f$ is the fluid viscosity. Moreover, $\eps$ denotes the (row-wise)
symmetric gradient, $\div$ is the divergence, $\grad$ is the gradient,
and $\partial_t$ denotes the (continuous) time-derivative.

The three field formulation \eqref{eq:biot:a}-\eqref{eq:biot:c} combines one 
scalar, time-dependent partial differential equation and two, stationary, 
vector partial differential equations.  This combination of time-dependent and 
time-independent equations can lead to non-trivial issues when considering 
discretizations of the time derivative; as a result: 
several splitting scheme approaches have been proposed~\cite{brun2019monolithic,%
girault2019priori,hong2020parameter,%
kumar2020conservative,storvik2019optimization}.    
In this manuscript we will focus on a monolithic approach, namely a 
straightforward backward Euler scheme, where all unknowns are solved for 
simultaneously. In the case of monolithic time discretization schemes:  
robustness with respect to material parameters in spatial
discretizations of \eqref{eq:biot} is a central concern and has been the topic 
of several recent investigations; %
c.f.~ e.g.~\cite{hong2017parameter,hu2017nonconforming,kraus2020parameter,rodrigo2018new}. A
notable difficulty, both practically and theoretically,
is that the parameter $\lambda$ may be very large, while $\kappa$ may be very 
small. The former corresponds to the (nearly) incompressible regime, while the 
latter corresponds to the (nearly) impermeable regime. Special care is
required in the formulation and analysis of discretizations
of~\eqref{eq:biot} to retain stability and convergence within these
parameter ranges.

Thus far, authors have analyzed mixed discretizations of \eqref{eq:biot} in the 
nearly-incompressible, and nearly-impermeable parameter regimes separately.  
For instance, a mixed discretization based on a total-pressure formulation  
\cite{kumar2020conservative,lee2018,lee2017parameter,lee2019,oyarzua2016locking}
has been well-studied and addresses the case of $\lambda\rightarrow \infty$.   
In the context of vanishingly small hydraulic conductivity, the concept of a 
\emph{Stokes-Biot stable} discretization has emerged 
\cite{hong2017parameter,lee2018,lotfian2018,rodrigo2018new} as a guide for the design of 
discrete schemes that retain their convergence properties as 
$\kappa \rightarrow 0$.

\begin{remark}
It is worth noting that the term \emph{Stokes-Biot stability} refers, in the 
contemporary literature, to a particular type of dual inf-sup condition.  It 
does not allude to an interface problem; that is, one should not confuse this 
term with that of a `\emph{Stokes-Darcy problem}', which refers to coupled 
Stokes and Darcy flow at an interface. 
\end{remark}
 
\subsection{An intuition for the Stokes-Biot stability condition} 

To motivate an intuitive view on the current notion of Stokes-Biot stability, 
we begin by considering a three-field variational formulation of a related 
system of (time-independent) equations: find $u \in U$, $z \in W$, and $p \in Q$ such that
\begin{subequations}
  \begin{alignat}{3}
    \label{eq:biot:steady:a}
    \inner{\sigma(u)}{\eps(v)} + \inner{\div v}{p} &= \inner{f}{v}
    &&\quad \foralls v \in U, \\
    \label{eq:biot:steady:b}
    \tau \kappa^{-1} \inner{z}{w} + \tau \inner{\div w}{p} &= \tau \inner{g}{w}
    &&\quad \foralls w \in W,\\  
    \label{eq:biot:steady:c}
    \inner{\div u}{q} +
    \inner{\tau \div z}{q} - \inner{c_0 p}{q} &= \inner{\tau s + \div \bar{u} - c_0 \bar{p}}{q},
    &&\quad \foralls q \in Q, 
  \end{alignat}
  \label{eq:biot:steady}%
\end{subequations}
for given $f, g, s, \bar{u}, \bar{p}$ and with $\inner{\cdot}{\cdot}$
denoting the standard $L^2$-inner product over the domain
$\Omega$. The continuous formulation~\eqref{eq:biot:steady} is representative 
of the equations resulting from an implicit Euler time discretization of~\eqref{eq:biot} 
with time step $\tau > 0$ and a prescribed set of homogeneous boundary conditions.  %
%
%
To continue, set $\tau = 1$; in this case we refer %
to~\eqref{eq:biot:steady} as (a mixed variational formulation of) a steady 
equation of Biot type; that is, the left-hand side is free of any time 
derivatives.  %
The system \eqref{eq:biot:steady} forms a generalized saddle-point
system which can be informally related \cite{lotfian2018} to a stand-alone 
Stokes-like, and stand-alone mixed-Darcy system.  For the former, multiply 
\eqref{eq:biot:steady:b} by $\kappa$, take $\kappa = c_0 = 0$ and assume 
$s = \div \bar{u} = 0$.  If $(u,z,p)$ solves~\eqref{eq:biot:steady} under 
these conditions, then $z = 0$ almost everywhere and
%
%
and~\eqref{eq:biot:steady} reduces to: find $u \in U$ and $q \in Q$
such that:
\begin{subequations}
  \begin{alignat}{3}
    \inner{\nu \eps(u)}{\eps(v)} + \inner{\div v}{p} &= \inner{f}{v}, \\
    \inner{\div u}{q} &= \inner{\div \bar{u}}{q} = 0, 
  \end{alignat}
  \label{eq:stokes}%
\end{subequations}
for all $v \in U$ and $q \in Q$, with $\nu = 2 \mu$. On
the other hand, if $c_0 = 0$ and the solution $(u,z,p)$ to \eqref{eq:biot:steady} 
satisfies $\div u = 0$ then $(z,p)$ solve the mixed Darcy problem: find 
$z\in W$ and $p \in Q$ such that
\begin{subequations}
  \begin{alignat}{3}
    \label{eq:darcy:a}
    \inner{\kappa^{-1} z}{w} + \inner{\div w}{p} &= \inner{g}{w}, \\
    \label{eq:darcy:b}
    \inner{\div z}{q} &= \inner{\tilde{s}}{q}, 
  \end{alignat}
  \label{eq:darcy}%
\end{subequations}
for all $w \in W$ and $q \in Q$ for given $g, \tilde{s}$. These observations hint
at a close relationship between the Stokes equations, Darcy equations and
the (steady) Biot-like system \eqref{eq:biot:steady}.
%
With this background: the Stokes-Biot stability concept 
\cite{hong2017parameter,lee2018,lotfian2018,rodrigo2018new} introduces two
conditions for finite element discretizations $U_h \times W_h \times
Q_h$ of~\eqref{eq:biot} or~\eqref{eq:biot:steady}:
\begin{enumerate}[label=(\roman*)]
\item
  the displacement-pressure pairing $U_h \times Q_h$ is a stable pair,
  in the sense of Babu\v ska-Brezzi~\cite{brezzi1974existence}), for the
  incompressible Stokes equations~\eqref{eq:stokes},
\item
  the flux-pressure pairing $W_h \times Q_h$ is a stable pair for the
  mixed Darcy problem~\eqref{eq:darcy}.
\end{enumerate}

\subsection{The Darcy assumption of Stokes-Biot stability}

Stokes-Biot stable discrete schemes should retain their convergence properties 
even when $\kappa \rightarrow 0$.  Indeed, a-priori error estimates, in 
appropriate parameter-dependent norms, have been advanced for both 
non-conforming \cite{hong2017parameter,kraus2020parameter,lee2018} and conforming 
\cite{rodrigo2018new} discretizations of~\eqref{eq:biot} or~\eqref{eq:biot:steady} 
satisfying the Stokes-Biot conditions (i) and (ii).  Consider a numerical test 
with two closely-related choices of discrete spaces;  the finite element pairings 
$\CG_2^d\times RT_0 \times \DG_0$ (product space of continuous piecewise quadratic 
vector fields, lowest order Raviart-Thomas elements and piecewise constants) and 
$\CG_1^d \times RT_0 \times \DG_0$. The former pairing satisfies conditions (i) 
and (ii) above (for given $\kappa > 0$), and is observed to converge even for 
$\kappa \ll 1$, see e.g.~\cite{rodrigo2018new} or Table~\ref{tab:p2rt0p0}. 
The latter pairing, which violates condition (i), can easily fail to
converge when $\kappa$ is sufficiently small (c.f.~ \cite[Table 2.1]{rodrigo2018new} or 
\cite[Section 6]{lotfian2018}).  This numerical observation demonstrates that 
condition (i), Stokes stability, is indeed an integral player in discretizations 
of \eqref{eq:biot} that retain their convergence behaviour as $\kappa\rightarrow 0$. 
\begin{table}
  \caption{Relative approximation errors for the displacement
    $\|\tilde u - u_h\|_1/\|\tilde u\|_1$ (top three rows in each
    table), pressure $\| \tilde p - p_h \|/\| \tilde p \|$ (middle
    three rows) and flux $\|\tilde z - z_h\|_{\div}/\|\tilde
    z\|_{\div}$ (bottom three rows) for varying $\kappa$ on a series
    of uniform meshes $\triang_h$ with mesh size $h$. The last column
    `Rate' denotes the order of convergence using for the last two
    values in each row. The exact solutions $\tilde u, \tilde p,
    \tilde z$, defined in Section~\ref{sec:numerics}, were represented
    by continuous piecewise cubic interpolants in the error
    computations. Similar results were obtained for $\kappa = 10^{-2},
    10^{-6}, 10^{-10}$ (data not shown). $c_0 = 0$. (A): $U_h \times
    W_h \times Q_h = \CG_2^d(\triang_h) \times RT_0(\triang_h) \times
    \DG_0(\triang_h)$. (B): $U_h \times W_h \times Q_h =
    \CG_2^d(\triang_h) \times \CG_1^d(\triang_h) \times
    \DG_0(\triang_h)$. }
  \begin{subtable}{1.0\textwidth}
    \begin{center}
      \caption{$\CG_2^d \times RT_0 \times \DG_0$}
      \label{tab:p2rt0p0}
      \begin{tabular}{l|lllll|c}
        \toprule
        \diagbox[]{$\kappa$}{$h$} & 1/8 & 1/16 & 1/32 & 1/64 & 1/128 & Rate \\
        \midrule
        $10^{0}$  & \num{1.64e-01} & \num{4.45e-02} & \num{1.13e-02} & \num{2.84e-03} & \num{7.11e-04} & $2.0$ \\
        $10^{-4}$ & \num{1.64e-01} & \num{4.45e-02} & \num{1.13e-02} & \num{2.84e-03} & \num{7.11e-04} & $2.0$ \\
        $10^{-8}$ &  \num{1.64e-01} & \num{4.45e-02} & \num{1.13e-02} & \num{2.84e-03} & \num{7.11e-04} & $2.0$ \\
        \midrule
        $10^{0}$ & \num{4.00e-01} & \num{1.03e-01} & \num{5.05e-02} & \num{2.53e-02} & \num{1.26e-02} & $1.0$\\
        $10^{-4}$ & \num{2.08e02} & \num{1.62e01} & \num{1.13e00} & \num{7.61e-02} & \num{1.34e-02} & $2.5$ \\
        $10^{-8}$ & \num{2.50e02} & \num{2.41e01} & \num{2.52e00} & \num{2.81e-01} & \num{3.51e-02} & $3.0$ \\
        \midrule
        $10^{0}$ & \num{1.30e00} & \num{1.76e-01} & \num{6.51e-02} & \num{3.18e-02} & \num{1.59e-02} & $1.0$ \\
        $10^{-4}$ & \num{3.06e03} & \num{6.76e02} & \num{1.19e02} & \num{1.65e01} & \num{2.11e00} & $3.0$ \\
        $10^{-8}$ & \num{4.51e03} & \num{2.27e03} & \num{1.31e03} & \num{7.09e02} & \num{3.61e02} & $1.0$ \\
        \bottomrule
      \end{tabular}
    \end{center}
  \end{subtable}
  \begin{subtable}{1.0\textwidth}
    \begin{center}
      \vspace{1em}
      \caption{$\CG_2^d \times \CG_1^d \times \DG_0$}
      \label{tab:p2p1p0}
      \begin{tabular}{l|lllll|c}
        \toprule
        \diagbox[]{$\kappa$}{$h$} & 1/8 & 1/16 & 1/32 & 1/64 & 1/128 & Rate \\
        \midrule
        $10^{0}$  & \num{1.64e-01} & \num{4.45e-02} & \num{1.13e-02} & \num{2.84e-03} & \num{7.12e-04} & $2.0$ \\
        $10^{-4}$ & \num{1.64e-01} & \num{4.45e-02} & \num{1.13e-02} & \num{2.84e-03} &  \num{7.11e-04} & $2.0$ \\
        $10^{-8}$ & \num{1.64e-01} & \num{4.45e-02} & \num{1.13e-02} & \num{2.84e-03} & \num{7.11e-04} & $2.0$ \\ 
        \midrule
        $10^{0}$ & \num{1.51e02} & \num{1.95e01} & \num{5.50e00} & \num{2.65e00} & \num{1.34e00} & $1.0$ \\
        $10^{-4}$ & \num{2.44e02} & \num{2.33e01} & \num{2.42e00} & \num{2.75e-01} & \num{3.52e-02} & $3.0$ \\
        $10^{-8}$ & \num{2.50e02} & \num{2.41e01} & \num{2.52e00} & \num{2.82e-01} &  \num{3.56e-02} & $3.5$ \\
        \midrule 
        $10^{0}$ & \num{1.12e00} & \num{1.71e-01} & \num{7.28e-02} & \num{3.62e-02} & \num{1.81e-02} & $1.0$ \\
        $10^{-4}$ & \num{4.56e02} & \num{8.98e01} & \num{1.25e01} & \num{1.31e00} & \num{1.17e-01} & $3.5$ \\
        $10^{-8}$ & \num{4.87e02} & \num{1.24e02} & \num{2.85e01} & \num{6.59e00} & \num{1.57e00} & $2.1$ \\
        \bottomrule
      \end{tabular}
    \end{center}
  \end{subtable}
  \vspace{1em}
\end{table}

The importance of the Stokes stability condition is not surprising from a 
theoretical perspective.  In the early Stokes-Biot literature, condition (i) plays 
a formative role \cite{rodrigo2018new} in showing that Euler-Galerkin discretizations 
of \eqref{eq:biot} remain inf-sup stable as $\kappa\rightarrow 0$.  Conversely,  
the Darcy stability condition (ii) is used to construct a projection that 
facilitates an a-priori analysis; the condition is not used in the stability 
argument.  This raises the question: is condition (ii) necessary to %
%
%
%
guarantee convergence as $\kappa\rightarrow 0$?  This question is important; the 
Darcy stability condition can easily fail to hold uniformly in $\kappa \ll 1$, 
thereby placing the previous analytic projection technique on questionable grounds.  
This observation was implicitly noted by other authors; c.f.~for instance \cite[Rmk.~5]{hong2017parameter}. 
More precisely, the continuous mixed Darcy problem~\eqref{eq:darcy} does not 
satisfy the Babuska-Brezzi conditions~\cite{brezzi1974existence} with bounds
independent of $0 < \kappa \ll 1$ in the standard $H(\div) \times L^2$
norm. To compensate, permeability-weighted flux and pressure norms,
such as e.g.~$\kappa^{-1/2} H(\div)\times \kappa^{1/2} L^2$, have been
suggested as viable alternatives~\cite{hong2017parameter}. However,
resorting to a permeability-weighted pressure space is not entirely
satisfactory; the relation between~\eqref{eq:biot:steady} and the
Stokes equations~\eqref{eq:stokes}, resulting from $\kappa \rightarrow 0$, 
points at $p \in L^2$ rather than $p \in \kappa^{1/2} L^2$. 

Moreover,  numerical experiments demonstrate convergence of the pressure 
in the $L^2$-norm even for diminishing $\kappa$, see
e.g.~Table~\ref{tab:p2rt0p0} for the pairing $\CG_2^d \times RT_0
\times \DG_0$. Conversely, consider the pairing $\CG_2^d \times
\CG_1^d \times \DG_0$ which violates the Darcy condition (ii), for any
$\kappa > 0$, and thus does not satisfy the Stokes-Biot stability
conditions. However, numerical experiments with this pairing,
see~Table~\ref{tab:p2p1p0}, show the hallmark of Stokes-Biot stable 
schemes.  That is, they appear stable, with the displacement and pressure errors 
converging at comparable rates as for $\CG_2^d \times RT_0
\times \DG_0$, for small $\kappa$; this behaviour even holds when $c_0 = 0$. 
These observations call into question the precise role of the Darcy
stability assumption in conforming mixed finite element
discretizations of \eqref{eq:biot} or~\eqref{eq:biot:steady}.

\subsection{Stokes-Biot stability revisited}

In this manuscript, we advance a theoretical point.  Namely, that a full 
Darcy inf-sup assumption is not necessary and can be relaxed; at least in the 
case of conformal Euler-Galerkin discretizations of~\eqref{eq:biot} 
or~\eqref{eq:biot:steady}.  %
%
%
Instead, we will see is that the following two assumptions are key:
\begin{enumerate}[label=(\Roman*)]
\item
  the displacement-pressure pairing $U_h \times Q_h$ is a stable pair
  for the incompressible Stokes equations~\eqref{eq:stokes}; and that
\item
  the inclusion $\div W_h \subseteq Q_h$ holds.
\end{enumerate}
We return to, and formalize, these \emph{minimal Stokes-Biot stability conditions} 
in Section~\ref{sec:stokes-biot-stability-criteria}.  In practice, the class of 
minimally Stokes-Biot stable discretizations are a superset of Stokes-Biot stable 
discretizations; one could then naturally consider dropping the distinction and, 
instead, viewing Stokes-Biot stability from this alternative point of view. 
We will show that the relaxed conditions produce schemes that retain their stability 
and convergence properties, in appropriate norms, as $\kappa \rightarrow 0$; motivated 
by the literature in applied porous-media modeling, we also note that this  
holds true for applications where $0 \leq c_0 < 1$ is chosen independently of 
other parameters.

Our primary purpose in this manuscript is theoretical in nature.  After introducing 
the relaxed conditions, we comment on the inf-sup stability and advance an 
a-priori analysis that does not employ a Galerkin projection technique; thus 
avoiding either an implicit dependence on $\kappa^{-1}$ in any projection estimates  
or the problematic question of uniform inf-sup Darcy stability as 
$\kappa \rightarrow 0$.  Unlike some previous endeavors of convergence estimates, 
we will conduct our estimates in the full norm used \cite{hong2017parameter,rodrigo2018new} 
for the inf-sup stability.  In particular, we introduce the norms in which 
Euler-Galerkin schemes, satisfying the relaxed conditions, are well posed and 
we show that the corresponding a priori error convergence rates: hold
in the limit as $\kappa \rightarrow 0$; and coincide with canonically
expected rates for well known mixed three-field finite element
paradigms (e.g.~first order for discretizations using linear or
Raviart--Thomas type flux approximations, etc).  Our objective, in clarifying 
these nuanced issues, is to establish a more consistent theory of Stokes-Biot 
stable schemes and to demonstrate an alternative, but standard, approach for their 
convergence analysis; such a view may also lead to downstream advances in the 
design of more efficient numerical schemes.  %
The remainder of this manuscript is organized as follows: Section 
\ref{sec:prelim} describes basic spaces and notation that 
will be used throughout; Section \ref{sec:stokes-biot-stability-criteria} 
overviews the current view of Stokes-Biot stability 
\cite{hong2017parameter,lee2018,rodrigo2018new,lotfian2018}; Section 
\ref{sec:euler-galerkin:well-posed} introduces a slight relaxation on the 
Stokes-Biot stable conditions and recalls a well-posedness argument for 
Euler-Galerkin discrete schemes; Section \ref{sec:a-priori:full-model} 
is aimed at a priori estimates for discretizations satisfying the relaxed 
conditions; finally, Section \ref{sec:numerics} is a numerical example 
demonstrating the retention of convergence behaviour as $\kappa \rightarrow 0$.  

\begin{remark}
In this manuscript, we are concerned with discretizations that retain their 
stability and convergence as $\kappa \rightarrow 0$.  The case of 
$\lambda \rightarrow \infty$ has been investigated separately 
\cite{kumar2020conservative,lee2018,lee2017parameter,lee2019,oyarzua2016locking} 
by introducing a total pressure, $\hat{p} = \lambda \nabla \cdot u - p$, to achieve 
robustness with respect to $\lambda$ when $\kappa\approx 1$ is assumed.  This 
view is similar to Herrmann's method~\cite{herrmann1965elasticity}, where 
a `solid pressure' term $p_s = \lambda\nabla\cdot u$, for elasticity systems 
in primal form with $\mu \ll \lambda$.  One may wonder if these 
methods can be brought together in a conformal setting.  This has not yet been 
investigated in the literature, and this is not the question we investigate in 
this manuscript; our current focus is to further the understanding of Stokes-Biot 
stable discretizations.  
\end{remark}

\section{Notation and preliminaries}
\label{sec:prelim}

\subsection{Sobolev spaces and norms}
Let $\Omega \subset \R^d$ for $d = 1, 2, 3$ be an open and bounded
domain with piecewise $C^2$ boundary
\cite{lipnikov-2002,showalter-2000,zenisek1984}.  We will consider
discretizations of $\Omega$ by simplicial complexes of order $d$.  All
triangulations, $\mathcal{T}_h$ of $\Omega$, will be assumed to be
shape regular with the maximal element diameter, also referred to as
the mesh resolution or mesh size, of $\mathcal{T}_h$ denoted by $h$.

We let $L^2(\Omega; \R^d)$, $H(\div, \Omega)$ and $H^1(\Omega; \R^d)$ denote the 
standard Sobolev spaces of square-integrable fields over $\Omega$, fields with
square-integrable divergence, and fields with square-integrable
gradient, respectively, and define the associated standard norms
\begin{align*}
  &\norm{f}^2 = \inner{f}{f}, \\
  &\norw{f}{1}^2 = \inner{f}{f}_1 = \inner{f}{f} + \inner{\grad f}{\grad f}, \\
  &\norw{f}{\div}^2 = %
   \inner{f}{f}_{\div} = \inner{f}{f} + \inner{\div f}{\div f}. 
\end{align*}
with $\inner{\cdot}{\cdot}_{\Omega}$ denoting the standard
$L^2(\Omega)$-inner product.  We will frequently drop the arguments
$\Omega$ and $\R^d$ from the notation when the meaning is clear from
the context.  %
%
%
The notation $H^1_{\Gamma}(\Omega)$ represents those functions in
$H^1$ with zero trace on $\Gamma \subseteq \partial
\Omega$. Similarly, $H_{\Gamma}(\div, \Omega)$ denotes fields in
$H(\div, \Omega)$ with zero (normal) trace on $\Gamma \subseteq
\partial \Omega$ in the appropriate
sense~\cite{boffi-brezzi-fortin2013}. We also define the standard
space of square-integrable functions with zero average:
\begin{equation*}
  L^2_0(\Omega) = %
  \left\{ p\in L^2(\Omega) \quad | \quad \int_{\Omega} p \dx = 0 \right\}.
\end{equation*}

We will also use parameter-weighted norms. For a Banach space $X$ and
real parameter $\alpha > 0$, the space $\alpha X$ signifies $X$
equipped with the $\alpha$-weighted norm $\norw{f}{\alpha X} = \alpha
\norw{f}{X}$. Finally, for a coercive and continuous bilinear form $a
: V \times V \rightarrow \R$, we will also write
\begin{equation*}
  \| v \|_{a}^2 = a(v, v).
\end{equation*}

\subsection{Intersections and sums of Hilbert spaces}

Let $X \subset Z$ and $Y \subset Z$ be two Hilbert spaces with a
common ambient Hilbert space $Z$. The intersection space, denoted $X
\cap Y$, is a Hilbert space with norm
\begin{equation*}   
  \|x\|_{X \cap Y}^2 = \|x\|^2_X + \|x\|^2_Y .
\end{equation*}
For instance, to illustrate our notation, the norm on the intersection
space $\kappa^{-1/2} L^2 \cap H(\div)$ is given by
\begin{equation*}
  \| v \|_{\kappa^{-1/2} L^2 \cap H(\div)}^2 = %
  \| v \|_{\kappa^{-1/2} L^2}^2 + \| v \|_{H(\div)}^2 = %
  \kappa^{-1} \| v \|_{L^2}^2 + \| v \|_{H(\div)}^2 .
\end{equation*}

The sum space $X+Y$ is the set $\left\{ z=x+y \,\,|\,\, x\in X,\, y\in
Y \right\}$ equipped with the norm
\begin{equation*}   
  \|z\|^2_{X + Y} = \inf_{\substack{z=x+y \\ x \in X, y \in Y}}\|x\|_X^2 + \|y\|_Y^2, 
\end{equation*}
and is also a Hilbert space. 
See e.g.~\cite[Ch. 2]{lofstrom1976} for a further discussion of sum and intersection 
spaces.  

\subsection{Operators}
For a given time step size $\tau$, times $t^{m-1}$ and $t^m$ and
fields $u^m \approx u(t^m)$ and $u^{m-1} \approx u(t^{m-1})$, we will
make use of a discrete derivative notation
\begin{equation}
  \label{eq:disc-deriv-notation}
  \partial_{\tau} u^m = \frac{u^m - u^{m-1}}{\tau}.
\end{equation}

\subsection{Finite element spaces}

Now, suppose that $\Omega\subset \mathbb{R}^d$ is a polygonal and let
$C^k(\Omega)$ denote the space of $k$-continuously differentiable
functions defined on $\Omega$. Let $D \subseteq \Omega$ and let
$P^k(D) \subset C^{\infty}(D)$ denote the set of polynomials of total
degree $k$ defined on $D$.  Let $\T_h$ be a simplicial triangulation
of $\Omega$ and let $T\in\T_h$ be any simplex; we denote the
restriction of a function $f$ to $T\in \T_h$ by $f_T$.  The notation
for the Lagrange elements of order $k$ used here is then
\begin{equation}\label{eqn:pk-spaces-on-elements}
  \CG_k(\T_h) =
  \left \{ f \in C^0(\Omega) \quad \vert \quad
  f_{T} \in P^k(T), \quad \foralls T \in \T_h \right \}.
\end{equation}
The notation $\CG_k^d(\T_h)$ will be used to represent the $d$-dimensional (vector) 
Lagrange spaces in $\mathbb{R}^d$.  The discontinuous Galerkin spaces of order 
$k$ relax the overall continuity requirement of the Lagrange finite element 
spaces; 
they are defined by 
\begin{equation}\label{eqn:dg-spaces-on-mesh}
\DG_k(\T_h) = \left\{ f \in L^2(\Omega) \quad \vert \quad
f_{T} \in P^k(T) \quad \foralls T \in \T_h \right \}.
\end{equation} 
A comprehensive discussion on Lagrange and discontinuous Galerkin
elements and their interpolation properties can be found in
e.g.~\cite{GUERMONDERN} and \cite{riviereDG} respectively. We will
also make use of the Brezzi-Douglas-Marini and Raviart--Thomas finite
element spaces~\cite[Sec.~2.3]{boffi-brezzi-fortin2013}. Throughout the
rest of the manuscript we use the notation $\CG_k$, $\CG_k^d$, $\DG_k$, $BDM_k$
and $RT_k$ in reference to the spaces defined above; that is, we drop
the additional mesh domain specification. 

\subsection{Boundary and initial conditions}
\label{sec:bcs}


General boundary conditions for \eqref{eq:biot} start by considering 
two distinct, non-overlapping partitions of the $d-1$ dimensional boundary 
$\partial\Omega$.  The first, corresponding to the displacement, is 
$\partial \Omega = \overline{\Gamma_c} \cup \overline{\Gamma_t}$ and the second, 
corresponding to the pressure, is denoted 
$\partial \Omega = \overline{\Gamma_p} \cup \overline{\Gamma_f}$; the non-overlapping 
condition means $\Gamma_c \cap \Gamma_t = \emptyset$ and $\Gamma_p \cap \Gamma_f = \emptyset$.   
The general form of the typical boundary conditions are then expressed as 
%
\begin{equation}
  \begin{array}{llcrr}
    u = 0, & \text{ on } \Gamma_c, & \text{ and } & z \cdot n =0, & \text{ on } \Gamma_f,\\
    p = 0, & \text{ on } \Gamma_p, & \text{ and } & \hat{\sigma}(u,p)\cdot n = 0 & \text{ on } \Gamma_t,
  \end{array}
  \label{eq:bcs:general}
\end{equation}
where $\hat{\sigma}(u,p) = \sigma(u) + p\,I_d$ and $I_d$ is the $d\times d$ identity 
matrix.  We will consider a simplification of the boundary conditions, above.  
The simplification that we will consider is that which was studied in the  
defining work on Stokes-Biot stable discretizations   
\cite{hong2017parameter,hu2017nonconforming,lipnikov-2002,rodrigo2018new}.  These 
conditions take $\Gamma_f = \Gamma_c$ and $\Gamma_p = \Gamma_t$ with the $d-1$ dimensional Lebesgue 
measure $|\Gamma_c| > 0$.  Thus we have  
\begin{equation}
  \begin{array}{llcrr}
    u = 0, & \text{ on } \Gamma_c, & \text{ and } & z \cdot n =0, & \text{ on } \Gamma_c,\\
    p = 0, & \text{ on } \Gamma_t, & \text{ and } & \sigma(u)\cdot n = 0 & \text{ on } \Gamma_t,
  \end{array}
  \label{eq:bcs}
\end{equation}
%
%
%
%
%
%
%
%
Let $\eta(x,t)$ denote the fluid content with equation 
\begin{equation*}
\eta(x,t) = c_0 p(x,t) +  \div u (x,t).
\end{equation*}
We follow \cite{showalter-2000} and remark: that under appropriate regularity 
assumptions on the sources and initial data, (i.e.~source data in 
$C^{\alpha}(0,T;(\Omega))$ where $\alpha$ is the Biot-Willis coefficient, 
boundary data in $C^{\alpha}(0,T;L^2(\partial \Omega))$, initial fluid content 
$\eta(x,0)\in L^2(\Omega)$, etc), then there exists a unique solution to 
\eqref{eq:biot} satisfying the boundary conditions \cite{showalter-2000,zenisek1984}.  

\begin{remark}
A full discussion on regularity details for the 
source, initial and boundary data can be found in \cite[Theorem~1]{zenisek1984}, and 
\cite[Sections~3 and 4]{showalter-2000}.  We also note that the boundary 
conditions \eqref{eq:bcs} reflect a restriction that may not be practical for 
many applications.  These boundary conditions coincide with those initially 
considered in the Stokes-Biot stability literature (e.g.~\cite{rodrigo2018new}) and 
allow the key ideas behind the Stokes-Biot (respectively, minimal Stokes-Biot) 
conditions, discussed in Section~\ref{sec:stokes-biot-stability-criteria} (respectively, 
Section~\ref{sec:euler-galerkin:well-posed}), to be discussed simply.  A  
discussion of more general conditions can be found in Section~\ref{sec:concluding-remarks},  
and in e.g.~\cite{hong2017parameter}.
\end{remark}



\subsection{Material parameters}

To facilitate the analysis here, we will assume that the material parameters of 
\eqref{eq:biot:a}-\eqref{eq:biot:c}, i.e.~$\mu$, $\lambda$, $\alpha$, $\kappa$, and $c_0$, 
are constant in space and time.  %
%
%
For simplicity and without loss of generality 
we set $\alpha = 1$.  This view can either be
interpreted literally or as having divided \eqref{eq:biot:a}-\eqref{eq:biot:c} 
through by $\alpha$ to obtain rescaled material parameters.  Moreover, one need 
not look far \cite{vardakis2019,vardakis2018,hong2017parameter,li2013,lotfian2018,riviere2019,youngriviere2014}  
to find applications where $\kappa$ is small, and the storage coefficient $c_0$ 
varies over a wide range of values in the presence of only modest choices of 
$\lambda$.  For instance, the literature contains examples of low hydraulic 
conductivities where both $\lambda$ and $c_0$ are approximately unity \cite{hong2017parameter}; 
in various soft-tissues, $\lambda \approx 10^2$ and $c_0 \approx 10^{-5}$ have been 
used \cite{vardakis2019,vardakis2018}, in addition to $\lambda \approx 10^1$ or 
$\lambda \approx 10^3$ with $c_0 \approx 10^{-10}$ \cite{li2013,riviere2019}, and 
even $c_0 = 0$ \cite{lotfian2018,youngriviere2014}.  This wide variation in 
$c_0$, while $\lambda$ remains modest, can be due to several reasons: an ad-hoc 
modeling assumption; to simplify numerical methods when storage coefficients are 
near the limits of computing precision; or due to the fact that, especially in 
biological applications, measurements for certain parameters may be unavailable 
and values are often estimated, chosen, or substituted from those, of similar 
biological regime, for which reasonable parameter estimates are available.

This manuscript is only concerned with Stokes-biot stable discretizations;  
these discretizations are designed to retain their stability and convergence 
properties in the presence of diminished hydraulic conductivity.  Given the wide 
variety of storage coefficients which appear, in the applied literature, in the 
presence of values for $\lambda \in [10^1,10^3]$ we take the view here that 
$c_0$ and $\lambda$ are independent parameters; this is not to assert that 
the linear poroelasticity theory does not imply that $\lambda\rightarrow\infty$ 
as $c_0 \rightarrow 0$.  %
%
Rather, we do this to make a secondary, strictly-numerical observation: that 
Stokes-Biot stable schemes, and the relaxation we propose herein, also retain 
their stability and convergence properties as $\kappa \rightarrow 0$ for every 
fixed choice of $0 \leq c_0 < 1$.  We will therefore assume that 
$0 < \kappa \leq 1$ and $0 \leq c_0 < 1$ are fixed, but otherwise arbitrary, 
constants. 

\begin{remark}
The bilinear forms defined in Section~\ref{sec:stokes-biot-stability-criteria} 
are parameter-dependent.  Thus, the arguments advanced in this manuscript 
may potentially be extended to parameters that vary in space or time, provided 
they satisfy suitable regularity requirements to justify the requiste 
manipulations.  As parameter, or data, regularity is not the focus on the 
current work, we do not take up this issue herein; we belay the topic and 
consider constant (i.e.~constant $\mu$, $\lambda$ and arbitrary but fixed 
$0<\kappa\leq1$, and $0\leq c_0 \leq1$) parameters. 
%
%
\end{remark}

\section{The Stokes-Biot stability conditions for conforming %
Euler-Galerkin schemes}
\label{sec:stokes-biot-stability-criteria}

Combining the nature of~\eqref{eq:biot} with the boundary
conditions~\eqref{eq:bcs}, we define the spaces
\begin{align}
  \label{eq:biot:spaces}
  U = H^1_{\Gamma_c}(\Omega), \quad
  W = H_{\Gamma_c}(\div, \Omega), \quad
  Q = L^2(\Omega).
\end{align}
We consider the following variational formulation of~\eqref{eq:biot} over
the time interval $(0, T]$: for a.e. $t \in (0, T]$, find the
    displacement $u$, flux $z$ and pressure $p$ such that $u(t) \in
    U$, $z(t) \in Z$ and $p(t) \in Q$ satisfy
\begin{subequations}
  \begin{alignat}{3}
    \label{eq:biot:var:a}   
    a(u, v) + b(v,p) &= \inner{f}{v} &&\quad v \in V, \\
    \label{eq:biot:var:b}  
    c(z, w) + b(w,p) &= \inner{g}{w} &&\quad w \in W, \\
    \label{eq:biot:var:c}
    b(\partial_t u, q) + b(z, q) - d(\partial_t p, q) &= \inner{s}{q} &&\quad q \in Q. 
  \end{alignat}
  \label{eq:biot:var}
\end{subequations}
The bilinear forms in~\eqref{eq:biot:var} are given by:
\begin{equation}
  \begin{split}
  &a(u, v) = \inner{\sigma(u)}{\eps(v)}, \quad
  b(u, q) = \inner{\div u}{q},  \\
  &c(z, w) = \inner{\kappa^{-1} z}{w}, \quad
  d(p, q) = \inner{c_0 p}{q}.
  \end{split}
  \label{eq:biot:forms}
\end{equation}
As noted in \cite{rodrigo2018new}: the existence and uniqueness of a
solution $(u,z,p)$ to \eqref{eq:biot:var}, with continuous dependence
on $f$, $g$ and $s$, has been established by previous authors
\cite{lipnikov-2002,showalter-2000,zenisek1984}.

\begin{remark}\label{rmk:clampedbd}
If Dirichlet conditions are imposed for the displacement on the entire
boundary and thus the pressure is only determined up to a constant
(i.e.~if $\Gamma_c = \partial \Omega$) we instead let $Q = L_0^2$.
\end{remark}

%

\subsection{An Euler-Galerkin discrete scheme}

Following \cite{rodrigo2018new} we consider Euler-Galerkin
discretizations; i.e.,~conforming finite element spaces in space and
an implicit Euler in time, of \eqref{eq:biot:var}. Let $0 = t_0 < t_1
< \ldots < t_N=T$ be a uniform partition of the time interval $[0,T]$.
The constant time step is then $\tau = \tau_m = t^m - t^{m-1}$. For
the function $f(t, x)$, evaluation at $t^m$ is denoted by $f^m =
f(t^m, x)$, and similarly for $g$ and $s$. We define conforming
discrete spaces
\begin{equation}
  U_h \subset U, \quad W_h \subset W, \quad Q_h \subset Q.
\end{equation}

The Euler-Galerkin discrete scheme of Biot's equations then reads as
follows: for each time iterate $m \in \left \{1, 2, \dots, N \right
\}$, given $f^m$, $g^m$, $s^m$, $\div u_h^{m-1}$, and, iff $c_0 > 0$,
$p_h^{m-1}$, we seek $(u_h^m, z_h^m, p_h^m) \in U_h\times W_h\times
Q_h$ such that
\begin{subequations}
  \begin{align}
    a(u_h^m, v) + b(v, p_h^m) &= \inner{f^m}{v}, \label{eq:biot:full-var-disc:a}  \\
    \tau c(z_h^m, w) + \tau b(w,p_h^m) &= \tau \inner{g^m}{w}, \label{eq:biot:full-var-disc:b} \\
    b(\partial_{\tau} u_h^m, q) + b(z_h^m, q) - d(\partial_{\tau} p_h^m, q) &= \inner{s^m}{q},  
    \label{eq:biot:full-var-disc:c}
  \end{align}
  \label{eq:biot:full-var-disc}
\end{subequations}
for all $v \in U_h$, $w \in W_h$ and $q \in Q_h$, and where we have
made use of the discrete derivative
notation~\eqref{eq:disc-deriv-notation}.  

\subsection{The Stokes-Biot stability conditions}

The Stokes-Biot stability conditions were introduced independently, in 
slightly different contexts, by 
several authors \cite{hong2017parameter,lee2018,lotfian2018,rodrigo2018new} and  
%
guide the selection of discrete spaces,
$U_h \times W_h \times Q_h$, for \eqref{eq:biot:full-var-disc}. We
recall a succinct statement of the (conforming) Stokes-Biot stability conditions, %
used in analogous forms by all original authors \cite{hong2017parameter,lee2018,rodrigo2018new}, 
here for posterity:
\begin{definition}[c.f.~{\cite[Defn.~3.1]{rodrigo2018new}}]
  \label{defn:original-stokes-biot-stability}
  The discrete spaces $U_h \subset U$, $W_h \subset W$ and $Q_h
  \subset Q$ are called a Stokes-Biot stable discretization if and
  only if the following conditions are satisfied:
  \begin{enumerate}[label=(\roman*)]
  \item The bilinear form $a$, as defined by~\eqref{eq:biot:forms}, is bounded and coercive on $U_h$;
  \item The pairing $(U_h, Q_h)$ is Stokes stable;
  \item The pairing $(W_h,Q_h)$ is Darcy (Poisson) stable.
  \end{enumerate} 
\end{definition}
\noindent We remark that \cite{hong2017parameter,lee2018}  were not 
conforming. More precisely, the
Stokes and Darcy stability assumptions of
Definition~\ref{defn:original-stokes-biot-stability} entail that the
relevant discrete spaces are stable in the (discrete) Babu\v%
ska-Brezzi sense~\cite{boffi-brezzi-fortin2013,brezzi1974existence}
for the discrete Stokes and Darcy problems, respectively. We will now
examine the Darcy stability condition more closely.

\subsection{The Darcy stability condition}
\label{subsec:stokes-biot-stability-criteria:darcy-stability-condition}
The discrete Darcy problem reads as: find $(z_h, p_h) \in W_h \times
Q_h$ such that~\eqref{eq:darcy} holds for all $w \in W_h$ and $q \in
Q_h$. Assume that $W \subseteq W$ and $Q_h \subset Q$ are equipped
with norms $\|\cdot\|_{W}$ and $\|\cdot\|_{Q}$, respectively. The
space $W_h \times Q_h$ is Darcy stable in the (discrete) Babu\v
ska-Brezzi sense if the discrete Babu\v ska-Brezzi conditions are
satisfied, in particular, if there exists constants $\alpha > 0$ and
$\beta > 0$, independent of $h$, such that
\begin{align}
  \label{eq:darcy:coercivity}
  &c(w, w) \geq \alpha \|w\|_{W}^2 \quad \foralls w \in \ker b = \{ w \in W_h\, | \, b(w, q) = 0 \, \foralls q \in Q_h\}, \\
  \label{eq:darcy:infsup}
  &\inf_{q \in Q_h} \sup_{w \in W_h} \frac{b(w, q)}{\| w \|_W \| q \|_Q} \geq \beta > 0,
\end{align}
with $b$ and $c$ as defined by~\eqref{eq:biot:forms}. It is also
assumed that $b$ and $c$ are continuous over $W \times Q$ and $W
\times W$ with respect to the relevant norms; i.e.~there exist
constants $C_b > 0$ and $C_c > 0$, independent of $h$, such that
\begin{equation}
  b(v, q) \leq C_b \| v \|_W \| q \|_Q, \quad
  c(v, w) \leq C_c \| v \|_W \| w \|_W,
\end{equation}
for all $v, w \in W$, $q \in Q$.

The assumption of discrete Darcy stability, and thus the existence of
solutions to the discrete Darcy problem, has been used to define
Galerkin projectors for use in the a-priori analysis of the Biot
equations~\eqref{eq:biot:full-var-disc}
(c.f.~for instance \cite[Sec.~4.2]{rodrigo2018new}). Given $z(t) \in W$ and $p(t) \in Q$
solving the continuous Biot equations~\eqref{eq:biot:var}, these
projectors $\Pi_{W_h} z(t)$ and $\Pi_{Q_h} p(t)$ solve the discrete
Darcy problem~\eqref{eq:darcy} for all $w \in W_h$, $q \in Q_h$ with
right-hand sides given by $\inner{g}{w} = c(z(t), w) + b(w, p(t))$ and
$\inner{s}{q} = b(z(t), q)$. For an a-priori analysis based on such a
Galerkin-projection approach to be optimal, including in the limit as
$\kappa \rightarrow 0$, the continuity constants $C_b, C_c$ and the
Babu\v ska-Brezzi stability constants $\alpha, \beta$ must be
independent of $0 < \kappa \leq 1$.

Attaining $\kappa$-independent continuity and stability constants is
non-trivial for the Darcy problem, and the norms that are selected for
$W$ and $Q$ play a vital role. For instance, the standard pairing
$H(\div) \times L^2$ with the natural norms is not appropriate as
e.g.~$c$ is not continuous with respect to the $H(\div)$ norm: the
continuity bound $C_c$ depends on $\kappa$. However, the following
pairings for $W \times Q$ are all meaningful for \eqref{eq:darcy} or
its dual $L^2 \times H^1$ formulation:
\begin{itemize}
\item[(A)] $\left ( \kappa^{-1/2} L^2 \cap H(\div) \right ) \times \left ( L^2 + \kappa^{1/2} H^1 \right )$      
\item[(B)] $\kappa^{-1/2} H(\div) \times \kappa^{1/2} L^2$       
\item[(C)] $\kappa^{-1/2} L^2 \times \kappa^{1/2} H^1$       
\end{itemize}

In particular, the inf-sup condition~\eqref{eq:darcy:infsup} holds
with inf-sup constant $\beta$ independent of $\kappa$ for each of
these pairings. We remark that $\|p\|_{L^2 + \kappa^{1/2} H^1} \le
\|p\|$ and $\|p\|_{\kappa^{1/2} L^2} \le \|p\|$ for $\kappa \le
1$. The $\kappa$-independent inf-sup condition for (A) was recently
shown in \cite{baerland2018uniform}, the inf-sup condition for (B)
follows directly by a scaling of the flux by $\kappa^{-1/2}$ and the
pressure by $\kappa^{1/2}$. Finally, the inf-sup condition of (C)
follows directly from Poincare's inequality with a similar scaling as
in (B). The boundedness of $b(z,p)$ can be established for 
each of the pairings above.  The pairing of (C) corresponds to the case of the 
$L^2\times H^1$ formulation of the mixed Darcy problem, 
i.e.~$b(z,p) = \inner{z}{\grad p}$ with $z\in W = L^2$ and $p\in Q=H^1$, but 
boundedness is proved in the same manner as for (B).  In the case of (B): 
applying Cauchy-Schwarz and the weighted norm definitions immediately gives 
\[
|b(z,p)|
\leq
\norm{\div z} \norm{p} \leq \norw{z}{H(\div)} \norm{p} = 
\left(\kappa^{-1/2}\norw{z}{H(\div)}\right)\left(\kappa^{1/2}\norm{p}\right).
\]
The case of (A) is complicated by the definition of the sum norm on
the pressure space $Q$, and a one-line argument is not possible
without additional context; see \cite{baerland2018uniform} for
details.

%
%

Options (A) and (B) above fit naturally with the variational
formulation of~\eqref{eq:biot:full-var-disc} and
spaces~\eqref{eq:biot:spaces}.  In the following, we suggest that a
natural norm for the Darcy flux is
\begin{equation} 
  \label{eq:W-norm}
  \Wnorm{z}^2 = \frac{\tau}{\kappa}\inner{z}{z} + \tau^2 \inner{\div z}{\div z}, 
\end{equation}
which is equivalent to the norm of the flux in (A) above for the
relevant range of $\kappa$ when $\tau > 0$. However, both options (A)
and (B) have disadvantages. For (B), the pressure norm (on $Q$)
becomes progressively weaker as $\kappa$ nears 0 while the norm of the
flux divergence (on $W$) is unnecessarily large compared with
e.g.~\eqref{eq:W-norm}. The primary drawback to using (A) is that the
pressure norm is implicitly defined. This fact means that an a-priori
analysis based on the method of projections is more complex to carry
out in practice; it is not clear that standard analytic techniques, e.g.~in 
\cite{hong2017parameter,hu2017nonconforming,lee2018,rodrigo2018new} among others, 
could be used directly when the norm of $L^2 + \kappa^{1/2} H^1$ is chosen for 
the pressure space.

We will argue instead that an a-priori analysis of
\eqref{eq:biot:full-var-disc} based on the use of a Galerkin
projection of the form \eqref{eq:darcy} is not necessary; thus
alleviating the need for an explicit uniform-in-$\kappa$ Darcy
stability condition on $(W_h,Q_h)$. Neither \eqref{eq:darcy:infsup}
nor the saddle-point stability of \eqref{eq:darcy} in general play a
role in the well-posedness of \eqref{eq:biot:full-var-disc}. Condition
(iii) of Definition~\ref{defn:original-stokes-biot-stability} will
thus be replaced by a less restrictive condition. An important
consequence of relaxing the uniform-in-$\kappa$ Darcy stability
hypothesis is that the standard $L^2$-norm on $Q$ can, and will, be
used.

\section{Minimal Stokes-Biot stability}
\label{sec:euler-galerkin:well-posed}

 
%
%

In this section we state the definition of minimal Stokes-Biot stability and 
recall a previous inf-sup condition in the spirit of the Banach-Ne\u{c}as-Babu\u{s}ka 
theorem.  In particular, the minimal Stokes-Biot stability conditions 
(c.f.~Definition~\ref{defn:minimal-stokes-biot-stability}) relinquish the Darcy 
stability assumption in favor of a containment condition.  In practice, this 
containment condition is satisfied for discrete flux-pressure pairings that are 
Darcy stable, though other discrete spaces satisfy this condition which are 
not stable pairings for the mixed Darcy problem.  %
%
%
%
%
%
Throughout this section we assume that $U$, $W$ and $Q$ are defined by
\eqref{eq:biot:spaces}. The norm on $U$ is taken to be the usual
$H^1(\Omega)$-norm $\| \cdot \|_1$, the norm on $Q$ is the standard
$L^2$-norm $\| \cdot \|$, while the norm $\Wnorm{\cdot}$ on $W$ is the
weighted norm defined by \eqref{eq:W-norm}. The norm \eqref{eq:W-norm}
was first introduced in \cite[Sec.~3.1]{hu2017nonconforming}. The
bilinear forms $a, b, c, d$ are as defined by~\eqref{eq:biot:forms}.

\subsection{Minimal Stokes-Biot conditions}

We now introduce our set of minimal Stokes-Biot stability
conditions. For clarity and completeness (rather than e.g.~brevity), we
include the precise stability conditions in the definition here. In
essence, between Definitions~\ref{defn:original-stokes-biot-stability}
and~\ref{defn:minimal-stokes-biot-stability}, only condition $(iii)$
changes.
\begin{definition}
  \label{defn:minimal-stokes-biot-stability}
  A family of conforming discrete spaces $\{ U_h \times W_h \times Q_h
  \}_h$ with $U_h \subset U$, $W_h \subset W$ and $Q_h \subset Q$ is
  called minimally Stokes-Biot stable if and only if
  \begin{enumerate}[label=(\roman*)]
  \item The bilinear form $a$ is continuous and coercive on $U_h
    \times U_h$; i.e.~there exists constants $C_a > 0$ and $\gamma_a >
    0$ independent of $h$ such that
    \begin{equation}
      \label{eq:a:bounds}
      a(u, u) \geq \gamma_a \norw{u}{1}^2,
      \quad a(u, v) \leq C_a \norw{u}{1} \norw{v}{1},
      \quad \foralls u, v \in U_h.
    \end{equation}
  \item The pairings $\{ U_h \times Q_h \}_h$ are Stokes stable in the
    discrete Babu\v ska-Brezzi %
    sense~\cite{braess2002finite,brezzi1974existence}; i.e.~in particular
    there exists an inf-sup constant $\beta_S > 0$ independent
    of $h$ such that
    \begin{equation}
      \label{eq:stokes:infsup}
      \inf_{q \in Q_h} \sup_{v \in U_h} %
      \frac{b(v, q)}{\norw{v}{1} \norm{q}} \geq \beta_S > 0.
    \end{equation}
  \item $\div W_h \subseteq Q_h$ for each $h$.
  \end{enumerate} 
\end{definition}

The classical flux-pressure pairings, e.g.~$RT_k \times DG_k$ or 
$BDM_{k+1}\times DG_k$ for $k=0,1,2,\dots$, satisfying 
Definition~\ref{defn:original-stokes-biot-stability}(iii) also satisfy the 
conditions of minimal Stokes-Biot stability; in particular 
Definition~\ref{defn:minimal-stokes-biot-stability}(iii).  However, the minimal 
Stokes-Biot condition also includes discretizations which are not encompassed 
by Definition~\ref{defn:original-stokes-biot-stability}.  For instance: 
flux-pressure pairings where the flux is taken from the space of continuous 
Lagrange polynomials can satisfy 
Definition~\ref{defn:minimal-stokes-biot-stability} while not satisfying 
Definition~\ref{defn:original-stokes-biot-stability}.   An illustration of this 
can be found in the family of discretizations where the displacement-pressure 
pairing are of Scott-Vogelius type; these either have 
the form $P^d_k \times RT_m \times DG_{k-1}$ or $P^d_k \times P^d_m \times DG_{k-1}$ 
where $k \geq 4$ and $0\leq m \leq k-1$.  The flux-pressure pairings 
$RT_m \times DG_{k-1}$, for $m < k-1$, and $P^d_m \times DG_{k-1}$, for $m\leq k-1$, 
are \textit{not} Darcy stable but \textit{do satisfy} the minimal %
Stokes-Biot stability containment condition of %
Definition~\ref{defn:minimal-stokes-biot-stability}(iii).  
%

A more pragmatic example 
is the discretization $P^d_2 \times RT_0 \times DG_0$.  This discretization is
both Stokes-Biot stable and minimally Stokes-Biot stable; of note is
that $P^d_2 \times P^d_1 \times DG_0$ is not Stokes-Biot stable but is
minimally Stokes-Biot stable.  The $P^d_2\times RT_0 \times DG_0$
discretization is a prototype for the minimal-dof displacement enrichment of 
a $P^d_1\times RT_0 \times DG_0$ approach studied in \cite{rodrigo2018new}.  %
%
%
%
The comparison between $P^d_2 \times RT_0 \times DG_0$ and 
$P^d_2 \times P^d_1 \times DG_0$ serves as a motivation for  
Definition~\ref{defn:minimal-stokes-biot-stability}, and will be studied in 
Section~\ref{sec:numerics}.  A further discussion of spaces that satisfy the 
minimal Stokes-Biot stability condition is given in 
Section~\ref{sec:concluding-remarks}.
%

\subsection{An inf-sup condition for minimal Stokes-Biot stable Euler-Galerkin schemes}
\label{sec:euler-galerkin:well-posed:subsec:proof}

The variational problem~\eqref{eq:biot:full-var-disc} can be shown to
satisfy a requirement of the Banach-Ne\u{c}as-Babu\u{s}ka theorem with respect 
to the weighted norm \eqref{eq:W-norm} and Definition~\ref{defn:minimal-stokes-biot-stability}.  
In fact, this result was proved in \cite{hu2017nonconforming}.

\begin{proposition}[Theorem 1, \cite{hu2017nonconforming}]\label{prop:relaxed-Stokes-Biot:stability:stab-prop} 
Let 
$\norm{(u_h,w_h,q_h)}_{UWQ}$ be defined by  
\[
\norm{(u_h,w_h,q_h)}_{UWQ} = \left( \norm{u_h}^2_1 + \Wnorm{w_h}^2 + \norm{q_h}^2 \right)^{1/2} 
\]
where $\Wnorm{w_h}$ is defined by \eqref{eq:W-norm}.  Define a composite 
bilinear form, on $U_h\times W_h \times Q_h$ and corresponding to \eqref{eq:biot:full-var-disc}, 
by the formula
\begin{align*}
\mathcal{B}(u_h,z_h,p_h;v_h,&r_h,q_h) = a(u_h,v_h) + b(v_h,p_h) + %
\tau~c(z_h,r_h) \\&+ \tau~b(r_h,p_h) + b(u_h,q_h) + %
\tau~b(z_h,q_h) - d(p_h,q_h)
\end{align*}
Suppose $U_h \times W_h \times Q_h$ satisfy the assumptions of 
Definition~\ref{defn:minimal-stokes-biot-stability}.  Then $\mathcal{B}$ is 
continuous and there exists a constant $\gamma > 0$, independent of 
$\kappa$ and $c_0$, such that 
\[
\sup_{(v_h,r_h,q_h)\in U_h\times W_h \times Q_h} \frac{\mathcal{B}(u_h,z_h,p_h;v_h,r_h,q_h)}{\norm{(v_h,r_h,q_h)}_{UWQ}} \geq \gamma \norm{(u_h,z_h,p_h)}_{UWQ} 
\]  
\end{proposition}
\begin{proof}
The proof follows from the arguments in \cite[Theorem 1]{hu2017nonconforming}.
\end{proof}

\begin{remark}
Work by previous authors \cite{hong2017parameter,hu2017nonconforming} shows that 
the assumptions of Definition~\ref{defn:minimal-stokes-biot-stability} were 
nascent in the literature.  The proof \cite{hu2017nonconforming} of 
Proposition~\ref{prop:relaxed-Stokes-Biot:stability:stab-prop}
is independent of $0\leq c_0$, and does not invoke Darcy stability, but does, 
in fact, use condition (iii) of Definition 
\ref{defn:minimal-stokes-biot-stability}.  
%
%
%
%
%
In fact, another version of Proposition~\ref{prop:relaxed-Stokes-Biot:stability:stab-prop} 
was also proved, independently, in \cite[Theorem 3.2, Case I]{hong2017parameter}; the proof, 
once more, is independent of $c_0$ and does not assume that the divergence maps 
the flux space surjectively onto the pressure space (i.e.~Darcy stability).  A 
nice mention of the case $U = H_0^1$ and $Q = L_0^2$ can also be found therein.  
The arguments of \cite[Theorem 3.2, Case I]{hong2017parameter} follow similarly 
to  those of \cite[Theorem 2]{hu2017nonconforming}.  
\end{remark}

\begin{corollary}
Assume that the assumptions of Definition~\ref{defn:minimal-stokes-biot-stability}
hold; then \eqref{eq:biot:full-var-disc} is well posed.
\end{corollary}
\begin{proof}
The Banach-Ne\u{c}as-Babu\u{s}ka theorem \cite[Theorem 2.6]{GUERMONDERN}, 
applied to \eqref{eq:biot:full-var-disc}, requires that two conditions are 
satisfied.  The first condition is that of 
Proposition~\ref{prop:relaxed-Stokes-Biot:stability:stab-prop}, which has been 
proved, independently, by several authors.  The second condition, which remains 
to be verified, is that if an element $(v_h,r_h,q_h)\in U_h\times W_h \times Q_h$  
is such that 
\[
\mathcal{B}(u_h,z_h,p_h;v_h,r_h,q_h) = 0,\quad \foralls (u_h,z_h,p_h) \in U_h\times W_h \times Q_h,  
\]
then $v_h = r_h = q_h = 0$ must follow.  To show that this condition also holds true, fix 
$(v_h,r_h,q_h) \in U_h\times W_h \times Q_h$ and suppose that the above implication  
holds; we need to show that, in this case, it must be that $v_h = r_h = q_h = 0$.  
Towards this end we consider two cases: the first case is if $c_0 > 0$, and the 
second case is if $c_0 = 0$. For the first case, select $u_h = v_h$, 
$z_h = r_h$ and $p_h = -q_h$, along with \eqref{eq:W-norm}, the hypothesis above 
and \eqref{eq:biot:forms}, to get 
\[
 \mathcal{B}(v_h,r_h,-q_h;v_h,r_h,q_h) = a(v_h,v_h) + \frac{\tau}{\kappa} \norm{r_h}^2 + c_0 \norm{q_h}^2 = 0. 
\]
Coercivity (c.f.~\eqref{eq:a:bounds}) gives 
$\gamma_a \norm{v_h}_1^2 \leq a(v_h,v_h)$ and $v_h = r_h = q_h = 0$ 
follows.  For the second case, assume that $c_0 = 0$.  The Stokes stability 
assumption (Definition~\ref{defn:minimal-stokes-biot-stability}(ii)) implies 
that (e.g.~\cite[p.~136]{braess2002finite}) there exists $y_h \in U_h$ such that
\begin{align}
	\inner{\div y_h}{q_h} &= \norm{q_h}^2, %
    	\label{eq:stability-proof:Stokes-inf-sup:a} \\
	\beta_S \norw{y_h}{1} &\leq \norm{q_h} %
	\label{eq:stability-proof:Stokes-inf-sup:b},
\end{align}
where $\beta_S$ is the Stokes inf-sup constant of \eqref{eq:stokes:infsup}.  
Let $\delta \geq 0$ be a yet-undetermined, but fixed, constant and choose 
$u_h = v_h + \delta y_h$, $z_h = r_h$, and $p_h = -q_h$. With these choices, 
and \eqref{eq:biot:forms}, we have
\[
 \mathcal{B}(v_h + \delta y_h,r_h,-q_h;v_h,r_h,q_h) = a(v_h,v_h) + 
\delta a(y_h,v_h) + \frac{\tau}{\kappa}\norm{r_h}^2 + \delta \norm{q_h}^2 = 0.
\]
The coercivity and continuity assumptions (c.f.~\eqref{eq:a:bounds}), together 
with Cauchy's inequality with epsilon, \eqref{eq:stability-proof:Stokes-inf-sup:b} 
and gathering of like terms gives
\begin{equation}
  \left(\gamma_a - \delta C_a \epsilon \right) \norw{v_h}{1}^2 +
  \frac{\tau}{\kappa} \norm{r_h}^2 + 
  \delta \left( 1- \frac{C_a}{4 \beta_S^2 \epsilon} \right)
  \norm{q_h}^2 \leq 0.
\end{equation}
We can now select the appropriate constants $\delta$ and $\epsilon$, as e.g.
\begin{equation}
  \epsilon = 2 \frac{C_a}{4 \beta_S^2} > 0,
  \quad
  \delta = \frac{\gamma_a \beta_S^2}{C_a^2} > 0,
\end{equation}
from which it follows that
\begin{equation}
  \label{eq:stability:first-inequal}
  \gamma_a \norw{v_h}{1}^2 +
  \frac{\tau}{\kappa} \norm{r_h}^2 + 
  \frac{1}{2} \frac{\gamma_a \beta_S^2}{C_a^2}
  \norm{q_h^m}^2 \leq 0.
\end{equation}
and thus $v_h = r_h = q_h = 0$.  Thus, the second condition of the 
Banach-Ne\u{c}as-Babu\u{s}ka theorem \cite[Theorem 2.6]{GUERMONDERN} holds, 
irregardless of $c_0$, and the result follows.
\end{proof}

\section{A priori error estimates for minimally Stokes-Biot stable schemes}
\label{sec:a-priori:full-model}


In this section, we derive a-priori error estimates for the
Euler-Galerkin discrete Biot equations~\eqref{eq:biot:full-var-disc} using the 
assumptions of Definition~\ref{defn:minimal-stokes-biot-stability}.  The final 
result is summarized in Proposition~\ref{cor:conv-est-specialized} of
Section~\ref{sec:a-priori:full-model:conv-estimates}.  We will assume
the point of view of minimal Stokes-Biot stability as defined
by~Definition~\ref{defn:minimal-stokes-biot-stability} and that $U_h$
contains the continuous nodal Lagrange elements $P_r$ for some $r \geq
1$. We begin by establishing basic assumptions on the spaces
$U_h, W_h$ and $Q_h$, and define projection operators in
Section~\ref{sec:a-priori:full-model:interpolants}.


\subsection{Projections and approximability}%
\label{sec:a-priori:full-model:interpolants}

As in the previous section, let $U$, $W$, $Q$ be given by
\eqref{eq:biot:spaces} with norms $\| \cdot \|_1$, $\Wnorm{\cdot}$
cf.~\eqref{eq:W-norm}, and $\| \cdot \|$, respectively. Assume that
the discrete spaces $U_h \times W_h \times Q_h$ satisfy the
assumptions of Definition~\ref{defn:minimal-stokes-biot-stability}. We
denote the (continuous) solutions to \eqref{eq:biot:var} at time $t^m$
by $(u^m, z^m, p^m)$ for $m = 1, 2, \dots, N$ while $(u_h^m, z_h^m,
p_h^m)$ represent the solutions of the discrete problem
\eqref{eq:biot:full-var-disc}.  %
For use in the subsequent error analysis, we make basic assumptions on
the spaces $U_h, W_h$ and $Q_h$, and define projection operators
$\Pi_{U_h}: U \rightarrow U_h$, $\Pi_{W_h} : W\rightarrow W_h$ and
$\Pi_{Q_h}: Q\rightarrow Q_h$ as follows. 

\begin{itemize}
  \item[$Q_h$:]
    Define $\Pi_{Q_h}$ to be the standard $L^2$-projection into
    $Q_h$. Then
    \begin{equation*}
      \| q - \Pi_{Q_h} q \| \ls \inf_{q_h \in Q_h}\| q - q_h\|,
    \end{equation*}
    for all $q \in Q$. If $Q_h$ contains piecewise polynomials of
    order $k = k_Q \geq 0$, then in particular
    \begin{equation}
      \label{eq:proj-pressure-ineq}
      \norm{q - \Pi_{Q_h}q} \ls h^{k_Q+1} \norw{q}{k_Q+1}, %
		\quad \foralls w \in H^{k}.
    \end{equation}
  \item[$W_h$:]
    Assume that $W_h$ contains (at least) piecewise polynomial
    (vector) fields of order $k = k_W \geq 0$.  We assume the
    existence of a generic discrete interpolant $\Pi_{W_h}: W
    \rightarrow W_h$ satisfying either
    \begin{equation}
      \label{eq:proj-commut-div-flux-ineq}
      \begin{array}{lcr}
        \norm{w - \Pi_{W_h} w} \ls h^{k_W + 1} \norw{w}{k_W + 1} & \text{and} & \norm{\div(w - \Pi_{W_h} w)} \ls h^{k_W+1} \norw{\div w}{k_W+1},
      \end{array}
    \end{equation}
    for $w \in H^{k_W+2}$, or
    \begin{equation}
      \label{eq:proj-commut-div-flux-ineq-lagrange}
      \begin{array}{lcr}
        \norm{w - \Pi_{W_h}w} \ls h^{k_W+1} \norw{w}{k_W+1} & \text{and} & \norw{w - \Pi_{W_h} w}{1} \ls h^{k_W}\norw{w}{k_W+1}.
      \end{array}
    \end{equation}
    for $w \in H^{k+1}(\Omega)$. The estimates
    \eqref{eq:proj-commut-div-flux-ineq} are characteristic of a
    Raviart-Thomas type, $RT_{k}$ ($k = 0, 1, 2, \dots$), interpolant
    whereas \eqref{eq:proj-commut-div-flux-ineq-lagrange} could
    correspond to a continuous Lagrange interpolant of order $k \geq
    1$~\cite{GUERMONDERN}.
  \item[$U_h$:]
    Following \cite{rodrigo2018new}, we define $\Pi_{U_h} : U \rightarrow U_h$ as a
    modified elliptic projection satisfying for $u \in U$:
    \begin{equation}
      \label{eq:proj-ellip-disp}
      a(\Pi_{U_h} u, v) = a(u, v) + b(v, q - \Pi_{Q_h} q) \quad \foralls v \in U_h, 
    \end{equation}
    where $q \in Q$ is given and will, in practice, be selected as the
    exact pressure solution to \eqref{eq:biot:var} at given times.

    Assume that $U_h$ contains (at least) continuous piecewise
    polynomial (vector) fields of order $k_U \geq 1$. There then
    exists an interpolant, $I^{k_U}: U \rightarrow U_h$, such that
    \begin{equation*}
      \norw{u - I^{k_U} u}{1} \ls h^{k_U} \norw{u}{k_U+1} 
    \end{equation*}
    for all $u \in H^{k_U+1}$, c.f. e.g~\cite{GUERMONDERN}.  Then for
    $u \in U$ we have
    \begin{equation*}
      \norw{u- \Pi_{U_h} u}{1}
      \leq \norw{u-I^{k_U} u}{1} + \norw{I^{k_U} u - \Pi_{U_h}u}{1} 
      \ls h^{k_U} \norw{u}{k_U+1} +  \norw{I^{k_U} u - \Pi_{U_h}u}{1}.
    \end{equation*}
    Using assumption (i) of
    Definition~\ref{defn:minimal-stokes-biot-stability} and
    \eqref{eq:proj-ellip-disp} with $v = \Pi_{U_h} u - I^{k_U} u$ imply
    that
    \begin{align*}
      \gamma_a \norw{\Pi_{U_h} u - I^{k_U} u}{1}^2 & \leq a(\Pi_{U_h} u - I^{k_U} u,\Pi_{U_h} u - I^{k_U} u) \\
      & =  a(u - I^{k_U} u, \Pi_{U_h} u - I^{k_U} u) + b(\Pi_{U_h} u - I^{k_U} u, q -\Pi_{Q_h}q)\\
      & \leq \norw{\Pi_{U_h} u - I^{k_U} u}{1} \left( C_a \norw{u - I^{k_U} u}{1} + \norm{q - \Pi_{Q_h}q} \right).
    \end{align*}
    Combining the above with assumption \eqref{eq:proj-pressure-ineq} gives
    \begin{equation}
      \label{eq:proj-ellip-disp-ineq}
      \norw{u - \Pi_{U_h} u}{1} \ls h^{k_U} \norw{v}{k_U + 1} + h^{k_Q+1}\norw{q}{k_Q+1}, 
    \end{equation}
    where $q \in Q$ is the fixed function defining the elliptic projection
    \eqref{eq:proj-ellip-disp}.
\end{itemize}

\subsection{Interpolation notation and identities}
\label{subsec:interpolation-and-approximation}

Following standard notation
\cite{hong2017parameter,lee2018,rodrigo2018new}, the error at time $t^m>0$ can be decomposed into 
interpolation errors $\rho$ and approximation errors $e$:
\begin{equation}
  \begin{split}
  u^m - u^m_h &= \left(u^m - \Pi_{U_h} u^m\right) - \left(u_h^m - \Pi_{U_h} u^m\right) = \rho_u^m - e_u^m \\
  z^m - z^m_h &= \left(z^m - \Pi_{W_h} z^m\right) - \left(z_h^m - \Pi_{W_h} z^m\right) = \rho_z^m - e_z^m\\
  p^m - p^m_h &= \left(p^m - \Pi_{Q_h} p^m\right) - \left(p_h^m - \Pi_{Q_h} p^m\right) = \rho_p^m - e_p^m.
  \end{split}
  \label{eq:err-analy-decomp}
\end{equation}
The interpolation errors satisfy the following identities. Since $\div
W_h \subseteq Q_h$ and by the definition of the $L^2$-projection
$\Pi_{Q_h}$, we have that
\begin{align}
  \label{eq:can:1}
  b(w, \rho_p^m) = \inner{\div w}{p^m - \Pi_{Q_h} p^m} = 0 \quad \foralls w \in W_h.
\end{align}
Similarly, by the definition of $\Pi_{Q_h}$,
\begin{align}
  \label{eq:can:2}
  d(\partial_{\tau} \rho_p^m, q) = c_0 \inner{\partial_{\tau} \rho_p^m}{q} = 0 \quad \foralls q \in Q_h,
\end{align}
where we recall the discrete derivative
notation~\eqref{eq:disc-deriv-notation}. Finally,~\eqref{eq:proj-ellip-disp}
directly gives
\begin{align}
  \label{eq:can:3}
  a(\rho_{u}^m, v) + b(v, \rho_{p}^m) = 0,  \quad \foralls v \in U_h.
\end{align}


Taking the difference between the continuous equations
\eqref{eq:biot:var} and discrete scheme \eqref{eq:biot:full-var-disc},
after multiplying \eqref{eq:biot:var:b} by $\tau$, combined with the
cancellations~\eqref{eq:can:1}--\eqref{eq:can:3}, yield the following
error equations at $t^m$: $(e_u^m, e_z^m, e_p^m)$ satisfies
\begin{subequations}
  \begin{align}
    \label{eq:a-priori:full-model:err-eq:a}  
    a(e_u^m, v) + b(v, e_p^m) &= 0
    && \foralls v_h \in U_h, \\
    \label{eq:a-priori:full-model:err-eq:b}  
    \tau c(e_z^m, w) + \tau b(w, e_p^m) &= \tau c(\rho_z^m, w),
    && \foralls w \in W_h, \\
    \label{eq:a-priori:full-model:err-eq:c}
    b(\partial_{\tau}e_u^m, q) + b(e_z^m, q) - d(\partial_{\tau}e_p^m, q) &= \inner{R^m}{q}
    && \foralls q \in Q_h,
  \end{align}
  \label{eq:a-priori:full-model:err-eq}  
\end{subequations}
where 
\begin{equation}
  \label{eq:first-ineq-residual}
  R^m = \div(\partial_t u^m - \partial_{\tau} u^m) + \div(\partial_{\tau}\rho_u^m) + \div \rho_z^m + c_0(\partial_t p^m - \partial_{\tau}p^m) ,
\end{equation}
by way of the general identity
\begin{equation}
  \label{eq:general-dt-dtau-identity}
  \partial_t u^m - \partial_{\tau} u_h^m = \partial_t u^m - \partial_{\tau} u^m  + %
  \partial_{\tau}\rho_{u}^m - \partial_{\tau} e_u^m,
\end{equation}
and similarly for $p$.

\subsection{Discrete approximation error estimates}
\label{sec:a-priori:full-model:error-estimates}

In this section we estimate the discrete errors described by
\eqref{eq:a-priori:full-model:err-eq} in their respective norms; that
is, $\norw{e_u^m}{1}$, $\Wnorm{e_z^m}$ and $\norm{e_p^m}$. In contrast
to e.g.~\cite{rodrigo2018new}, we do not make use of the restrictive
uniform-in-$\kappa$ Darcy stability assumption. In turn, the error
equations require a more technical analysis and we have adapted
related methods originally used to study $\kappa$ fixed 
\cite{lee2018} and vanishing ($c_0$) storage coefficient. Despite the more technical approach,
the resulting estimates presented in~Proposition
\ref{cor:conv-est-specialized} is directly comparable to related
results in the literature; c.f.~\cite[Lem.~3]{hong2017parameter},
\cite[Thm.~4.1]{lee2018} and \cite[Thm~4.6]{rodrigo2018new}. We
conclude that the concept of minimal Stokes-Biot stability provides
analogous error estimates for a more general set of conforming
discrete spaces than the original Stokes-Biot stability concept.

During the course of the analysis will make use of the following
useful inequality
\begin{lemma}{\cite[Lemma 3.2]{lee2018}}
  \label{lemma:lee}
  Suppose that $A$, $B$, $C$ $>0$ and $D\geq 0$ satisfy
  \[
  A^2 + B^2 \leq CA + D.
  \]
  Then either $A+B \leq 4C$ or $A+B \leq 2\sqrt{D}$ holds.
\end{lemma}

\begin{proposition}
  \label{prop:disc-errors}
  Suppose that $U_h \times W_h \times Q_h$ is minimally Stokes-Biot
  stable (by satisfying the assumptions of
  Definition~\ref{defn:minimal-stokes-biot-stability}). Then, the
  discrete approximation errors $(e_u^m, e_z^m, e_p^m)$ described by
  \eqref{eq:a-priori:full-model:err-eq} satisfy the inequality:
  \begin{multline}
      \norw{e_u^m}{1} + \norm{e_p^m} + \Wnorm{e_z^m}
      \ls \norw{e_u^0}{1} + \norw{e_p^0}{d} 
      + \tau^{1/2} \norw{e_z^0}{c}  \\
      + \left( \int_{0}^{T} \norw{\rho_z}{c}^2  \ds \right)^{1/2} 
      + \tau \int_0^T \norw{\rho_{\partial_t z}}{c}^2 \ds 
      + C_{\tau}^T ,
      \label{eq:prop:disc-errors}
  \end{multline}
  with inequality constant depending on $C_a$, $\gamma_a^{-1}$ and where 
  \begin{equation*}
    C_{\tau}^m
    \equiv \int_{0}^{t_m} \norm{\div \rho_z} + \norw{\rho_{\partial_t u}}{1} + \tau \left ( c_0 \norm{\partial_{tt}p} + \norw{\partial_{tt}u}{1} \right) \ds .
  \end{equation*}
\end{proposition} 
\begin{proof}
In an analogous fashion as for
Proposition~\ref{prop:relaxed-Stokes-Biot:stability:stab-prop},
multiplying \eqref{eq:a-priori:full-model:err-eq:c} by $\tau$,
selecting $v = e_u^m - e_u^{m-1}$, $w = e_z^m$, and $q = -e_p^m$ in
\eqref{eq:a-priori:full-model:err-eq} and summing gives
\begin{equation}
  \label{eq:first-ineq-basis}
  \begin{split}
  a(e_u^m - e_u^{m-1}, e_u^m) + \tau c(e_z^m,e_z^m) 
  + d(e_p^m - e_p^{m-1},e_p^m) 
  = \tau c(\rho_z^m,e_z^m) - \tau\inner{R^m}{e_p^m} .
  \end{split}
\end{equation}

For any (continuous) symmetric bilinear form $a$ with induced norm
$\norw{\cdot}{a}$ we have the inequality~\cite{ern-munier-2009}
\begin{equation}
  \label{eq:inequality:symmetric-bilinear-form}
  \frac12\left(\norw{\chi}{a}^2 - \norw{\chi-\xi}{a}^2\right) \leq a(\xi,\chi).
\end{equation}
Using the above, and the symmetry of both $a(\cdot,\cdot)$ and
$d(\cdot,\cdot)$, it follows that the left-hand side of
\eqref{eq:first-ineq-basis} is bounded below by
\begin{equation}\label{eq:a-priori:err-eq:lower-est-a}
  \frac12 \norw{e_u^m}{a}^2 - \frac12 \norw{e_u^{m-1}}{a}^2 + %
  \tau \norw{e_z^m}{c}^2 + \frac12 \norw{e_p^m}{d}^2 - %
  \frac12 \norw{e_p^{m-1}}{d}^2.
\end{equation}
On the other hand, Cauchy-Schwarz and Young's inequality give
\begin{equation}
  \label{eq:a-priori:err-eq:upper-est-a}
  |\tau c(\rho_z^m,e_z^m)| \leq \frac{\tau}{2} \norw{\rho_z^m}{c}^2  + \frac{\tau}{2} \norw{e_z^m}{c}^2. 
\end{equation}
From the Stokes stability assumption~\eqref{eq:stokes:infsup} and
\eqref{eq:a-priori:full-model:err-eq:a} we have the estimate
\begin{equation}
  \label{eq:err-est:stokes-inf-sup-est}
  \beta_S \norm{e_p^m} \leq \sup_{v \in U_h} \frac{b(v, e_p^m)}{\norw{v}{1}}
  = \sup_{v \in U_h}\frac{-a(e_u^m,v)}{\norw{v}{1}} \leq C_a \norw{e_u^m}{1}.
\end{equation}
Then, Cauchy-Schwarz, \eqref{eq:err-est:stokes-inf-sup-est} and the coercivity of $a$ gives
\begin{equation}\label{eq:a-priori:err-eq:upper-est-b}
  \tau |\inner{R^m}{e_p^m}| \leq  C_a \beta_S^{-1} \gamma_a^{-1/2} \tau \norm{R^m} \norw{e_u^m}{a}.
\end{equation}
Combining \eqref{eq:a-priori:err-eq:lower-est-a}, 
\eqref{eq:a-priori:err-eq:upper-est-a},
\eqref{eq:a-priori:err-eq:upper-est-b} yields
\begin{equation}
    \label{eq:a-priori:err-eq:summed-est-pre}
    \norw{e_u^m}{a}^2 - \norw{e_u^{m-1}}{a}^2 + \tau \norw{e_z^m}{c}^2 + \norw{e_p^m}{d}^2 - \norw{e_p^{m-1}}{d}^2
    \ls \tau \left (\norw{\rho_z^m}{c}^2 + \norm{R^m} \norw{e_u^m}{a} \right ).
\end{equation}
with inequality constant depending on $C_a \beta_S^{-1} \gamma_a^{-1/2}$. \\

\noindent \textbf{Estimate of $\norw{e_u^m}{a}$}: Following a
technique from \cite{lee2018}, let $J$ be the integer index where
$\norw{e_u^m}{a}$ (for $m = 1, \dots, N$) obtains its maximal
value. Summing \eqref{eq:a-priori:err-eq:summed-est-pre} from $m=1$ to
$m=J$, using the maximality assumption, and re-arranging terms yields
\begin{equation}
  \label{eq:a-priori:err-eq:summed-est-a}
  \norw{e_u^J}{a}^2 + \tau\sum\limits_{m=1}^{J}\norw{e_z^m}{c}^2 + 
  \norw{e_p^J}{d}^2 
  \ls \norw{e_u^0}{a}^2 + \norw{e_p^0}{c}^2 + \sum\limits_{m=1}^{J} \tau \norw{\rho_z^m}{c}^2 + \sum\limits_{m=1}^{J} \tau \norm{R^m} \norw{e_u^J}{a}.
\end{equation}
%
%
%
We can apply Lemma~\ref{lemma:lee} to \eqref{eq:a-priori:err-eq:summed-est-a} 
by taking $A = \norw{e_u^J}{a}$, $B = \norw{e_p^J}{d}$ and dropping the 
additional left-hand side term; then we choose 
\begin{equation*}
  C = \sum \limits_{m=1}^{J}\tau \norm{R^m}, \quad
  D =  \norw{e_u^0}{a}^2 + \norw{e_p^0}{c}^2 +\sum \limits_{m=1}^{J}\tau\norw{\rho_z^m}{c}^2,
\end{equation*}
%
and, provided appropriate temporal regularity of the exact solution, have
\begin{equation*}
\sum\limits_{m=1}^{J} \tau \norw{\rho_z^m}{c}^2 \ls \int_{0}^{t^J} \norw{\rho_z}{c}^2 \ds.
\end{equation*}
Lemma~\ref{lemma:lee}, with the above and the triangle inequality, implies
\begin{equation}
  \norw{e_u^J}{a} + \norw{e_p^J}{d}
  \ls \norw{e_u^0}{a} + \norw{e_p^0}{d} + \sum\limits_{m=1}^{J}\tau\norm{R^m} + \left(\int_{0}^{t^J} \norw{\rho_z}{c}^2\right)^{1/2}.
  \label{eq:a-priori:err-eq:summed-est-b}
\end{equation}

\noindent \textbf{Bound of $\tau \norm{R^m}$}:  We now develop a bound for the terms $\tau\norm{R^m}$; 
c.f.~\eqref{eq:first-ineq-residual}.  From the fundamental theorem of calculus 
and integration by parts we have the general result
\begin{equation*}
  \partial_t f^m - \partial_{\tau}f^m = \frac{1}{\tau}\int_{t^{m-1}}^{t^m} (s-t^{m-1})\partial_{tt}f(s) \ds 
\end{equation*}
for any $m = 1, \dots, N$, assuming sufficient temporal regularity of
the field $f$. We therefore, again under the assumption of sufficient
spatial and temporal regularity, have the inequalities
\begin{align}
  c_0 \norm{\partial_t p^m - \partial_{\tau}p^m}
  &\leq \int_{t^{m-1}}^{t^m} c_0 \norm{\partial_{tt} p} \ds
  \label{eq:residual-ineq-bound:a}\\
  \norm{\div\left(\partial_t u^m - \partial_{\tau}u^m\right)}
  &\leq \int_{t^{m-1}}^{t^{m}} \norw{\partial_{tt}u}{1} \ds,
  \label{eq:residual-ineq-bound:b}
\end{align}
which control the first and fourth terms of $\norm{R^m}$.

For the second term of $R^m$ we have $\norm{\div\partial_\tau
  \rho_u^m} \leq \norw{\partial_\tau \rho_u^m}{1}$. Rearranging the
terms of $\partial_{\tau}\rho_{u}^m$, applying the fundamental theorem
of calculus and using the commutation of the time derivative with the
elliptic projection \eqref{eq:proj-ellip-disp} yields
\begin{equation}
  \label{eq:residual-ineq-bound:c}
  \norw{\partial_{\tau}\rho_{u}^m}{1} = 
  \norw{\frac{u^m - u^{m-1}}{\tau} - \frac{\Pi_{U_h}u^m - \Pi_{U_h} u^{m-1}}{\tau}}{1} \leq 
  \frac{1}{\tau} \int_{t^{j-1}}^{t^j} \norw{\rho_{\partial_t u}}{1} \ds .
\end{equation} 
For the third term of $R^m$ we have, again up to sufficient temporal regularity of the exact solution, 
that
\begin{equation}
  \label{eq:residual-ineq-bound:d}
  \sum\limits_{m=1}^{J}\tau\norm{\div\rho_z^m} \ls %
  \int_{0}^{t^J} \norm{\div \rho_z} \ds .
\end{equation}
Summarizing, \eqref{eq:residual-ineq-bound:a}-\eqref{eq:residual-ineq-bound:d} %
thus yield
\begin{equation}
  \begin{split}
    \sum_{m=1}^J \tau \norm{R^m} \ls
    &\int_{0}^{t^J} \norm{\div \rho_z} + \norw{\rho_{\partial_t u}}{1} + %
    \tau \left ( c_0 \norm{\partial_{tt}p} + %
    \norw{\partial_{tt}u}{1} \right) \ds \equiv C_{\tau}^J.
  \end{split}
  \label{eq:R:bound}
\end{equation}
And so, the estimate \eqref{eq:a-priori:err-eq:summed-est-b} becomes
\begin{equation}
    \label{eq:u-err-est:pre-bound:final}
    \norw{e_u^J}{a} + \norw{e_p^J}{d}
    \ls \norw{e_u^0}{a} + \norw{e_p^0}{d} +
    \left(\int_{0}^{t^J} \norw{\rho_z}{c}^2 \ds \right)^{1/2}
    + C_\tau^J
\end{equation}
Clearly, by Definition~\ref{defn:minimal-stokes-biot-stability}(i),
this also gives a bound for $\norw{e_u^m}{1}$ (depending on
$\gamma_a^{-1}$) for $m = 1, \dots, N$. \\

\noindent \textbf{Estimate of $\norm{e_p^m}$}: The norm
$\norw{e_p^J}{d}$ in e.g.~\eqref{eq:u-err-est:pre-bound:final}
vanishes in the limit as $c_0 \rightarrow 0$. An alternative bound for
$\norm{e_p^m}$ can be derived from the Stokes stability assumption,
Definition~\ref{defn:minimal-stokes-biot-stability}(ii). In
particular, using \eqref{eq:err-est:stokes-inf-sup-est} and
\eqref{eq:u-err-est:pre-bound:final} it follows that for each $1 \leq
m \leq N$:
\begin{equation}
  \label{eq:err-est:ep-bounded-by-eu}
  \norm{e_p^m} \ls \norw{e_u^m}{1} \ls \norw{e_u^J}{a},
\end{equation}
with inequality constant $C_a \beta_S^{-1} \gamma_a^{-1/2}$ and where
$J$ is the index where $\norw{e_u^J}{1}$ is maximal. Thus
$\norm{e_p^m}$ can be bounded by the right hand side of
\eqref{eq:u-err-est:pre-bound:final}, independently of $c_0$. \\

%
%
\noindent \textbf{Estimate of $\tau\norw{e_z^m}{c}^2$}: In order to
estimate the flux error in the norm defined
by~\eqref{eq:W-norm}, i.e.~$\Wnorm{e_z^m}$, it will be advantageous to %
consider the constituents
separately; e.g.~$\tau\norw{e_z^m}{c}^2$ and $\tau^2 \norm{\div
  e_z^m}^2$.

We begin by considering the first component and again argue based on
maximality. Take the difference of the error
equation \eqref{eq:a-priori:full-model:err-eq:a} at time levels $m$,
$m-1$ and dividing by $\tau$ to get
\begin{equation}
  \label{eq:apriori:max-z:first-est}
  a(\partial_{\tau} e_u^m, v) + b(v, \partial_{\tau} e_p^m) %
  = 0 \quad \text{for } v \in U_h.
\end{equation}
Similarly taking the difference of
\eqref{eq:a-priori:full-model:err-eq:b} at time levels $m$ and $m-1$,
and divide by $\tau^2$ to get
\begin{equation*}
  c(\partial_{\tau} e_z^m, w) + b(w, \partial_{\tau} e_p^m) = %
  c(\partial_{\tau}\rho_{z}^m, w) \quad \text{for } w \in W_h.
\end{equation*} 

Choose $v = \partial_{\tau} e_u^m$, $w = e_z^m$ in the above as well
as $q = - \partial_{\tau} e_p^m$ in
\eqref{eq:a-priori:full-model:err-eq:c}; summing these three
equations, using Cauchy-Schwarz on the right-hand side, and coercivity
on the left-hand side gives
\begin{align*}
  \gamma_a \norw{\partial_{\tau} e_u^m}{1}^2  + %
  \norw{\partial_{\tau} e_p^m}{d}^2 + c(\partial_{\tau} e_z^m, e_z^m) 
  \leq \norw{\partial_{\tau}\rho_z^m}{c}\norw{e_z^m}{c} + %
  \norm{R^m} \norm{\partial_{\tau} e_p^m} .
\end{align*}
From Definition~\ref{defn:minimal-stokes-biot-stability}(ii) and
\eqref{eq:apriori:max-z:first-est} we have that
$\norw{\partial_{\tau}e_p^m}{} \leq C_a \beta_S^{-1}
\norw{\partial_{\tau}e_u^m}{1}$ by the analogue
of~\eqref{eq:err-est:stokes-inf-sup-est}. Using this on the right-most
term of the above, alongside Cauchy's inequality with epsilon and
choosing epsilon appropriately, gives
\begin{equation*}
  \norw{\partial_{\tau}e_u^m}{1}^2  + \norw{\partial_{\tau} e_p^m}{d}^2 + %
  c(\partial_{\tau} e_z^m, e_z^m) 
  \ls \norw{\partial_{\tau}\rho_z^m}{c} \norw{e_z^m}{c} + \norm{R^m}^2,
\end{equation*}
with inequality constant depending on $C_a \beta_S
\gamma_a^{-1}$. Dropping the positive displacement and pressure
left-hand side terms, multiplying both sides by $\tau$, and using the
symmetry of $c$ together with the inequality
\eqref{eq:inequality:symmetric-bilinear-form} give
\begin{equation*}
  \norw{e_z^m}{c}^2 - \norw{e_z^{m-1}}{c}^2
  \ls \tau \norw{\partial_{\tau} \rho_{z}^{m}}{c} \norw{e_z^m}{c} + \tau \norm{R^m}^2.
\end{equation*}

Let $M$ be the index where $\norw{e_z^m}{c}^2$ achieves its maximum
for $1 \leq m \leq N$. Summing the above from $m=1$ to $m=M$, using
the maximality of $\norw{e_z^M}{c}$, multiplying both sides by $\tau$
and re-arranging yields
\begin{equation}
  \label{eq:z-err-est:pre-bound:a}
  \tau \norw{e_z^M}{c}^2
  \ls \tau \norw{e_z^0}{c}^2
  + \tau \left(\sum\limits_{m=1}^{M} %
  \tau \norw{\partial_{\tau}\rho_{z}^m}{c}\right) \norw{e_z^M}{c} + %
  \sum\limits_{m=1}^{M} \left ( \tau \norm{R^m} \right )^2.
\end{equation}
By the fundamental theorem of calculus, we have 
\begin{equation*}
   \tau \left ( \norw{\partial_{\tau}\rho_z^m}{c} \right ) = %
   \norw{\rho_z^m - \rho_z^{m-1}}{c} %
   \leq \int_{t^{m-1}}^{t^m} \norw{\rho_{\partial_t z}}{c} \ds. 
\end{equation*}
Applying H{\"o}lder's inequality on the right-most term, above, gives
\begin{equation*}
\int_{t^{m-1}}^{t^m} \norw{\rho_{\partial_t z}}{c} \leq 
\left(\int_{t^{m-1}}^{t^m} 1\, dt\right)^{1/2}%
\left( \int_{t^{m-1}}^{t^m} \norw{\rho_{\partial_t z}}{c}^2\right)^{1/2} 
\end{equation*}
so that 
\begin{equation}\label{eq:z-err-est:pre-bound:b}
\tau \left ( \norw{\partial_{\tau}\rho_z^m}{c} \right ) \leq 
\tau^{1/2}\left( \int_{t^{m-1}}^{t^m} \norw{\rho_{\partial_t z}}{c}^2\right)^{1/2}. 
\end{equation}
Inserting \eqref{eq:z-err-est:pre-bound:b} and
\eqref{eq:R:bound} into
\eqref{eq:z-err-est:pre-bound:a}, using Young's inequality on the
second term on the right-hand side and rearranging yields
\begin{align}
  \tau \norw{e_z^M}{c}^2 &\ls
  \tau \norw{e_z^0}{c}^2 + %
  \tau\left(\tau\sum\limits_{m=1}^{M}\norw{\partial_{\tau}\rho_{z}^m}{c}\right)^2 
  + \sum\limits_{m=1}^{M} \left ( \tau \norm{R^m} \right )^2, \nonumber\\ 
  & \ls \tau \norw{e_z^0}{c}^2 + %
  \tau^2 \int_{0}^{t_M} \norw{\rho_{\partial_t z}}{c}^2 \ds + %
  \left( C_{\tau}^M \right)^2
  \label{eq:z-err-est:pre-bound:first-term-final}
\end{align}
%
%
\noindent\textbf{Estimate of $\tau^2\norw{\div e_z^m}{}$}:\\ Now we
estimate the second, and final, term in the flux norm
\eqref{eq:W-norm}. Let $K$ denote the index where
$\norm{\div e_z^K}$ is maximal.  Using
Definition~\ref{defn:minimal-stokes-biot-stability}(iii), and
selecting $q = \tau\div e_z^K$ in the error equation
\eqref{eq:a-priori:full-model:err-eq:c} for $m = K$ yields
\begin{equation*}
  \inner{\div (e_u^K - e_u^{K-1})}{\div e_z^K}
  + \tau \inner{\div e_z^K}{\div e_z^K}
  - \inner{c_0 (e_p^K - e_p^{K-1})}{\div e_z^K}
  = \tau \inner{R^K}{\div e_z^K}.
\end{equation*}
Thus, re-arranging terms, using Cauchy-Schwarz and the triangle
inequalities, and dividing by $\norm{\div e_z^K}$ gives
\begin{equation*}
  \begin{split}
    \tau \norm{\div e_z^K}
    &\ls \norw{e_u^K}{1} + \norw{e_u^{K-1}}{1} + 
    c_0 \norm{e_p^K} + c_0 \norm{e_p^{K-1}} + \tau \norm{R^K}\\
    &\ls \norw{e_u^J}{1} + \tau \norm{R^K}, 
  \end{split}
\end{equation*}
where the last inequality follows from the majorization of the terms
$e_u^K$, $e_p^K$, $e_u^{K-1}$, $e_p^{K-1}$ by the maximum $e_u^J =
\max\limits_{j=1,2,\dots,N}e_u^{j}$
and~inequality \eqref{eq:err-est:ep-bounded-by-eu}. Noting that
\begin{equation*}
  \left(\norw{e_u^J}{1} + \tau \norm{R^K} \right)^{2} \ls %
  \norw{e_u^J}{1}^2 + \tau^2 \norm{R^K}^2 \ls %
  \norw{e_u^J}{1}^2 + \sum_{m=1}^{K} \tau^2 \norm{R^m}^2, %
\end{equation*}
and employing \eqref{eq:u-err-est:pre-bound:final}, 
\eqref{eq:R:bound} and taking $I=\max\left\{J,K\right\}$ then 
gives
\begin{equation}
  \tau^2 \norm{\div e_z^K}^2
  \ls
  \norw{e_u^0}{1}^2 + \norw{e_p^0}{d}^2 +
  \int_{0}^{t_I} \norw{\rho_z}{c}^2 \ds + \left ( C_{\tau}^I \right)^2 .
  \label{eq:z-err-est:pre-bound:second-term-final}   
\end{equation} 
  
Finally, to establish \eqref{eq:prop:disc-errors}, combine the
definition of the weighted flux norm \eqref{eq:W-norm},
\eqref{eq:u-err-est:pre-bound:final},
\eqref{eq:err-est:ep-bounded-by-eu}
\eqref{eq:z-err-est:pre-bound:first-term-final}, and %
\eqref{eq:z-err-est:pre-bound:second-term-final} and use the fact that
the integral from $0$ to $T$ majorizes all of the time-integral
right-hand sides of the summed expressions. 
\end{proof}


\subsection{Convergence estimates}%
\label{sec:a-priori:full-model:conv-estimates}
To specialize the general results of
Proposition~\ref{prop:disc-errors} we will first suppose the exact
solutions to \eqref{eq:biot:var:a}-
suitable regularity assumptions.  Moreover, we assume the
interpolants, discussed in \ref{sec:a-priori:full-model:interpolants},
satisfy approximation inequalities of a certain order.  Towards that
end let $U_h \times W_h \times Q_h$ satisfy the assumptions of
Definition \ref{defn:minimal-stokes-biot-stability}.  %
For a reflexive Banach space $X$, a time interval $(a, b) \subseteq \R$ and
a measurable $f: (a,b) \rightarrow X$ we define the canonical
space-time norm \cite{evans10}
\begin{equation}
  \label{eq:standard-space-time-norm}
  \norw{f}{L^p(a,b;X)} = \left( \int_{a}^{b} \norw{f(s)}{X} \ds \right)^{1/p}. 
\end{equation}
As in the case of spatial derivatives, the usual Sobolev notation $f
\in H^r(a, b; X)$ means that $f \in L^2(a, b; X)$ and that $\partial_t
f$, $\partial^2_t f$, $\dots$, $\partial^r_t f$ are also in $L^2(a, b;
X)$. In the sections that follow we will sometimes use
  the abbreviations $\norw{f}{L^2 X}$ or $\norw{f}{H^r X}$ to signify
  \eqref{eq:standard-space-time-norm}.

\begin{proposition}\label{cor:conv-est-specialized}
Suppose the assumptions of Proposition~\ref{prop:disc-errors} hold.  %
Let $k \geq 0$ be the greatest integer such that the orthogonal projection, 
$\Pi_{Q_h}: Q \rightarrow Q_h$, satisfies \eqref{eq:proj-pressure-ineq}.  
Suppose $r \geq 1$ is the maximal integer such that $P_r$, the space 
of continuous Lagrange polynomials of order $r$, is contained in $U_h$; 
suppose an interpolation, from $W$ to $W_h$, satisfying either 
\eqref{eq:proj-commut-div-flux-ineq} or 
\eqref{eq:proj-commut-div-flux-ineq-lagrange} exists and let $s > 0$ be the 
maximal integer satisfying the respective inequality. Suppose that the exact 
solutions to 
\eqref{eq:biot:var:a}-\eqref{eq:biot:var:c} satisfy the regularity 
assumptions
\begin{equation*}
\begin{array}{lll}
u(t) \in L^{\infty}((0,T];H^{r+1}\cap U) %
& \partial_t u\in L^{1}((0,T];H^{r+1}\cap U) %
& \partial_{tt}u\in L^{1}((0,T];H^1))\\
\multicolumn{2}{l}{z(t) \in  %
L^{\infty}((0,T];H^{s+1}\cap W)\cap L^{\infty}((0,T];H^{s+1}_{\kappa^{-1}}\cap W)} %
& \partial_t z \in L^2((0,T];H^{s+1}_{\kappa^{-1}}\cap W)\\
\multicolumn{2}{l}{p(t) \in L^{\infty}((0,T];H^{k+1}\cap Q)} %
& \partial_{tt}p \in L^1((0,T];L^1),
\end{array}
\end{equation*}
and that the initial iterates, $(u_h^0,z_h^0,p_h^0)$, satisfy the estimates  
\begin{align} 
\norw{u(0)-u_h^0}{1} & + \tau^{1/2} \norw{z(0)-z_h^0}{c} + \norw{p(0)-p_h^0}{d}
\label{eqn:initial-data-approx} \\
\ls &\,\, h^{r}\norw{u(0)}{H^{r +1}} + 
\tau^{1/2}h^{s+1}\norw{z(0)}{\kappa^{-1}H^{s + 1}} + 
h^{k+1}\norw{p(0)}{H^{k + 1}},\nonumber
\end{align}
consistent with the projections of %
section~\ref{sec:a-priori:full-model:interpolants}.  %
Then for $c = \min\{k,r,s\}$ we have
\begin{equation}\label{eq:conv-est}
\norw{u^m-u_h^m}{1} + \Wnorm{z^m - z_h^m} + \norm{p^m - p_h^m} \ls h^c M_1 + \tau M_2
\end{equation} 
where $M_1$ and $M_2$ are given by 
\begin{align*}
M_1 &= h^{r-c}\left(\norw{\partial_t u}{L^1H^{r+1}} + %
\norw{u}{L^{\infty}H^{r+1}}\right) + %
h^{s-c}\left(\norw{z}{L^1H^{s+1}}\right.  \\
& \quad + \left. %
(h+\tau^{1/2})\norw{z}{L^{\infty}H^{s+1}_{\kappa^{-1}}} \right)%
+ h^{k-c}\norw{p}{L^{\infty}H^{k+1}}\\
M_2 &= c_0 \norw{\partial_{tt} p}{L^1L^1} + %
\norw{\partial_{tt}u}{L^1H^1} + %
h^{s+1}\norw{\partial_t z}{L^2H^{s+1}_{\kappa^{-1}}} + %
h^s\norw{z}{L^{\infty}H^{s+1}} 
\end{align*}
\end{proposition}
\begin{proof}
First, note that since $\Pi_{Q_h}$ satisfies \eqref{eq:proj-pressure-ineq} 
and since $P_r \subset U_h$ then, according to the argument directly preceding 
\eqref{eq:proj-ellip-disp-ineq}, the inequality 
\eqref{eq:proj-ellip-disp-ineq} holds.  Using the triangle inequality, 
\begin{align*}
\norw{e_u^0}{1} + \tau^{1/2} \norw{e_z^0}{c} + \norw{e_p^0}{d} & \leq 
\norw{u(0)-u_h^0}{1} + \tau^{1/2}\norw{z(0)-z_h^0}{c} \\ 
& \quad + \norw{p(0)-p_h^0}{d} 
+ \norw{\rho_u^0}{1} + \norw{\rho_z^0}{c} + \norw{\rho_p^0}{d}, 
\end{align*}
along with \eqref{eqn:initial-data-approx} and the projection estimates of 
section \ref{sec:a-priori:full-model:interpolants}, applied to the last three terms 
above, gives 
\begin{equation}\label{eqn:initial-data-approx:disc-errors}
\norw{e_u^0}{1} + \tau^{1/2} \norw{e_z^0}{c} + \norw{e_p^0}{d} \ls 
\,\, h^{r}\norw{u(0)}{H^{r +1}} + 
\tau^{1/2}h^{s+1}\norw{z(0)}{\kappa^{-1}H^{s + 1}} + h^{k+1}\norw{p(0)}{H^{k + 1}}.
\end{equation}
Then \eqref{eq:conv-est} follows from the triangle inequality, with respect to 
the error decompositions \eqref{eq:err-analy-decomp},  along with: the discrete 
error estimates \eqref{eq:prop:disc-errors}; discrete initial iterate error 
estimates \eqref{eqn:initial-data-approx:disc-errors}; and interpolation estimates \eqref{eq:proj-pressure-ineq}, %
\eqref{eq:proj-commut-div-flux-ineq}-%
\eqref{eq:proj-commut-div-flux-ineq-lagrange} and \eqref{eq:proj-ellip-disp}. 
\end{proof}

\begin{remark}
Further assumptions on the discrete spaces, beyond the minimal Stokes-Biot 
stability of Defn.~\ref{defn:minimal-stokes-biot-stability}, can lead to 
slightly different versions of Proposition \ref{cor:conv-est-specialized}.  
For instance, if $(W_h,Q_h)$ are such that the usual Raviart-Thomas type 
projection commutation relation   
\[
\left(\div z - \div\Pi_{W_h}z,q_h = 0\right),\quad \text{for all } q_h \in Q_h,
\]
holds for each $z\in W$ then $(\div \rho_z,q_h) = 0$ so that, for instance, the 
contribution $\norw{z}{L^1 H^{s+1}}$ vanishes from $M_1$; this term arises 
from $\norw{\div\rho_z}{}$ in \eqref{eq:prop:disc-errors}.  This observation 
is used in \cite{rodrigo2018new} where $W_h = RT_0$ is fixed.
\end{remark}

\section{Numerical experiments}
\label{sec:numerics}
Turning to numerical evaluation of our theoretical findings, we
investigate the stability and numerical convergence properties, for $0
< \kappa \ll 1$ and $0 \leq c_0 \leq 1$, of two mixed finite element
pairings:
\begin{enumerate}[label=(\roman*)]
\item
  $U_h \times W_h \times Q_h = \CG^2_2 \times RT_0 \times \DG_0$, and
\item
  $U_h \times W_h \times Q_h = \CG^2_2 \times \CG^2_1 \times \DG_0$
\end{enumerate}
The first discretization is a canonical choice from the original view
of conforming Stokes-Biot \cite{hu2017nonconforming,rodrigo2018new}
stability whereas the second choice is only minimally Stokes-Biot
stable. We define a manufactured, smooth exact solution set over the
unit square $\Omega = [0,1] \times[0,1]$, with coordinates $x = (x_1,
x_2) \in \Omega$, given by
\begin{equation*}
  u(t, x) = 
  \begin{pmatrix}
    t\sin(\pi x_1) \sin(\pi x_2) \\
    2 t \sin(3 \pi x_1) \sin(4\pi x_2)
  \end{pmatrix},
  \quad
  p(t, x) = (t+1) \left (
  \left( (x_1 - 1) x_1 (x_2 - 1) x_2 \right)^2 - \frac{1}{900}
  \right).
\end{equation*}
for $t \in (0, T)$, $T = 1.0$.
These solutions satisfy the homogeneous boundary conditions
$u_{|\partial \Omega} = 0$ and $z_{|\partial\Omega}\cdot n = 0$ where
$n$ is the outward boundary normal to the unit square. By construction
$p(t, \cdot) \in L^2_0(\Omega)$ for each $t$.
For each discretization we consider three parameter scenarios:
vanishing storage ($c_0 = 0$), fixed storage ($c_0 = 1$) and
diminishing hydraulic conductivity ($\kappa \rightarrow 0$), and fixed
hydraulic conductivity ($\kappa = 1.0$) and vanishing storage ($c_0
\rightarrow 0)$. For simplicity, we here consider unit Lam\'e
parameters: $\mu = \lambda = 1.0$. We let the time step size $\Delta t
= T = 1.0$ as the test case is linear in time. For
solving~\eqref{eq:biot:var} numerically, we used the FEniCS finite
element software suite~\cite{fenics-one}. The zero average-value
condition on the pressure is enforced via a single real Lagrange
multiplier. Linear systems were solved using MUMPS.

In Sections ~\ref{sec:numerics:Stokes-Biot-Standard} and %
\ref{sec:numerics:Stokes-Biot-Minimal} we examine the numerical errors for a 
fixed, minimally Stokes-Biot stable discretization on a series of uniform meshes, 
$\triang_h$, with mesh size $h$.  Each of the error tables in 
Sections~\ref{sec:numerics:Stokes-Biot-Standard} and \ref{sec:numerics:Stokes-Biot-Minimal} 
follow the same general format.  In general, the relative displacement errors 
$\|\tilde u(T) - u_h(T)\|_1/\|\tilde u(T)\|_1$ are reported in the first 
set of table rows; the relative pressure errors $\| \tilde p(T) - p_h(T) \|/\| \tilde p(T) \|$ 
appear in the second set of table rows; and the relative flux errors 
$\Wnorm{\tilde z(T) - z_h(T)}/\Wnorm{\tilde z(T)}$ appear in the final set of 
table rows.  The last column of each table (`Rate') denotes the order of 
convergence using for the last two values in each row.  For each investigation, 
either $c_0$ or $\kappa$ varies while the other is fixed; the result of the variable 
parameter is reported for the values $10^{-r}$ for $r=0,4,8$ and $12$ but intermediate 
results identical to the previous case are suppressed.  For instance, if 
$\kappa=10^0$, $\kappa=10^{-4}$, and $\kappa = 10^{-8}$ all yield the same errors 
for a given quantity, then only the errors for $\kappa=10^0$ and $\kappa = 10^{-12}$ 
are reported in the corresponding row. In many instances, the displacement errors 
correspond directly to those of a previous case; in this event, we refer to the 
appropriate table. 

\subsection{Convergence of a Stokes-Biot stable pairing}
\label{sec:numerics:Stokes-Biot-Standard}
We first consider the convergence properties for the pairing $U_h
\times W_h \times Q_h = \CG^2_2(\triang_h) \times RT_0(\triang_h)
\times \DG_0(\triang_h)$.  This discretization satisfies the minimal Stokes-Biot 
stability conditions and also the Darcy stability condition when $\kappa$ is 
uniformly bounded below.  We note, however, that the Darcy stability condition fails 
to hold \cite{baerland2018uniform} uniformly as $\kappa$ tends to zero but 
that Definition~\ref{defn:minimal-stokes-biot-stability}(iii) does indeed hold 
regardless of $\kappa$.  As discussed in the previous section, we report on the relative approximation
errors for the displacement, 
pressure, 
and flux 
for a series of uniform meshes $\triang_h$ with mesh size $h$. The exact solutions
$\tilde u, \tilde p, \tilde z$ were represented by continuous
piecewise cubic interpolants in the error computations.

  \subsubsection{Vanishing storage $c_0 = 0$, varying conductivity $0 < \kappa \leq 1$}
  \label{sec:numerics:Stokes-Biot-Standard:a}
  (see: Table~\ref{tab:numerics:p2rt0p0:A}) We observe that the
  displacement error converges at the expected and optimal rate (2)
  for $\kappa$ ranging from $1$ down to $10^{-12}$. Overall, the
  displacement errors remain essentially unchanged as $c_0$ and
  $\kappa$ vary. (We therefore do not report or discuss these further
  here.) The behaviour for the flux and pressure errors is less
  regular. The flux and pressure approximation errors increase as
  $\kappa$ decreases, but seem to stabilize i.e.~not increase
  substantially further from $\kappa = 10^{-4}$ to $10^{-8}$ and to
  $10^{-12}$. Moreover, for each $\kappa$, the pressure and flux
  errors decrease with decreasing mesh size. Indeed, for $h = 1/128$,
  the pressure errors are of similar magnitude for the range of
  hydraulic conductivities ($\kappa$) tested.  For a comparison to a minimally 
  Stokes-Biot stable analogue, see Section \ref{sec:numerics:Stokes-Biot-Minimal:a} 
  and Table~\ref{tab:numerics:p2p1p0:A}. 

  \subsubsection{Fixed storage $c_0 = 1$, varying conductivity $0 < \kappa \leq 1$} 
  \label{sec:numerics:Stokes-Biot-Standard:b}
  (see: Table~\ref{tab:numerics:p2rt0p0:B}) For this case, we again
  observe that the flux and pressure approximation errors increase as
  $\kappa$ decrease, but seem to stabilize and not increase
  substantially further from $\kappa = 10^{-4}$ to $10^{-8}$ and
  $10^{-12}$. Again, for each $\kappa$, the pressure and flux errors
  decrease with decreasing mesh size and for $h = 1/128$, the pressure
  errors are nearly identical for the range of hydraulic conductivites ($\kappa$) 
  tested.  For comparison, see Section \ref{sec:numerics:Stokes-Biot-Minimal:b} 
  and Table~\ref{tab:numerics:p2p1p0:B}.
  
  \subsubsection{Fixed conductivity $\kappa = 1$, varying storage $0 \leq c_0 \leq 1$} 
  \label{sec:numerics:Stokes-Biot-Standard:c}
  (see: Table~\ref{tab:numerics:p2rt0p0:C}) For this case, we
  observe nearly uniform behaviour as $c_0$ decreases. The pressure
  and flux errors are similar for the range of storage coefficients ($c_0$) considered, and
  converge at the optimal and expected rate (1). For comparison, see Section 
  \ref{sec:numerics:Stokes-Biot-Minimal:c} and Table~\ref{tab:numerics:p2p1p0:C} 

\begin{table}
\caption{Vanishing storage coefficient $c_0 = 0$, varying conductivity
  $0 < \kappa \leq 1$ for the (minimally) Stokes-Biot stable pairing 
  $\CG^2_2(\triang_h) \times RT_0(\triang_h) \times \DG_0(\triang_h)$. 
  Relative approximation erros for the time-dependent test case given in 
  Section~\ref{sec:numerics}.  Listed are the relative displacement (top), 
  relative flux (middle) and relative pressure (bottom) errors for varying 
  $\kappa$ on a series of uniform meshes $\triang_h$ with mesh size $h$.  
  The displacement errors for $\kappa = 10^{-4}, 10^{-8}$ were identical to 
  the data presented ($\kappa =1$, $\kappa = 10^{-12}$) and are suppressed.  
  The last column `Rate' denotes the order of convergence using for the last 
  two values in each row. Compare with Table~\ref{tab:numerics:p2p1p0:A}.}
  \label{tab:numerics:p2rt0p0:A}
  \begin{center}
    \begin{tabular}{l|lllll|c}
      \toprule
      \diagbox[]{$\kappa$}{$h$} & 1/8 & 1/16 & 1/32 & 1/64 & 1/128 & Rate \\
      	\midrule
	\multicolumn{7}{c}{Displacement}\\
	\midrule
      $10^{0}$  & \num{1.64e-01} & \num{4.45e-02} & \num{1.13e-02} & \num{2.84e-03} & \num{7.11e-04} & $2.0$ \\
      $10^{-12}$ & \num{1.64e-01} & \num{4.45e-02} & \num{1.13e-02} & \num{2.84e-03} & \num{7.11e-04} & $2.0$ \\ 
      	\midrule
	\multicolumn{7}{c}{Pressure}\\
      	\midrule
      $10^{0}$  & \num{2.63e-01} & \num{1.02e-01} & \num{5.05e-02} & \num{2.53e-02} & \num{1.26e-02} & $1.0$ \\
      $10^{-4}$ & \num{1.04e02} & \num{8.12e00} & \num{5.69e-01} & \num{4.39e-02} & \num{1.28e-02} & $1.8$ \\
      $10^{-8}$ & \num{1.25e02} & \num{1.21e01} & \num{1.26e00} & \num{1.42e-01} & \num{2.07e-02} & $2.8$ \\
      $10^{-12}$ &  \num{1.25e02} & \num{1.21e01} & \num{1.26e00} & \num{1.43e-01} & \num{2.09e-02} & $2.8$ \\
      	\midrule
	\multicolumn{7}{c}{Flux}\\
	\midrule	
      $10^{0}$ &  \num{6.88e-01} & \num{1.41e-01} & \num{6.39e-02} & \num{3.18e-02} & \num{1.59e-02} & $1.0$ \\
      $10^{-4}$ & \num{3.62e02} & \num{4.38e01} & \num{5.35e00} & \num{6.46e-01} & \num{8.05e-02} & $3.0$ \\
      $10^{-8}$ & \num{4.72e02} & \num{9.91e01} & \num{2.67e01} & \num{7.01e00} & \num{1.76e00} & $2.0$ \\ 
      $10^{-12}$ &  \num{4.72e02} & \num{9.91e01} & \num{2.67e01} & \num{7.04e00} & \num{1.79e00} & $2.0$ \\
      \bottomrule
    \end{tabular}
  \end{center}
\end{table}


\begin{table}
\caption{Fixed storage coefficient $c_0 = 1$, varying conductivity $0
  < \kappa \leq 1$ for the (minimally) Stokes-Biot stable pairing 
  $\CG^2_2(\triang_h) \times RT_0(\triang_h) \times
  \DG_0(\triang_h)$.  The format follows that of Table~\ref{tab:numerics:p2rt0p0:A} 
  and the relative displacement errors are identical.  Compare with 
  Table~\ref{tab:numerics:p2p1p0:B}}
  \label{tab:numerics:p2rt0p0:B}
  \begin{center}
    \begin{tabular}{l|lllll|c}
      \toprule
      \diagbox[]{$\kappa$}{$h$} & 1/8 & 1/16 & 1/32 & 1/64 & 1/128 & Rate \\
	\midrule
	\multicolumn{7}{c}{Displacement, c.f.~Table \ref{tab:numerics:p2rt0p0:A}}\\
	\midrule
	\multicolumn{7}{c}{Pressure}\\
      	\midrule
      $10^{0}$  & \num{2.61e-01} & \num{1.02e-01} & \num{5.05e-02} & \num{2.53e-02} & \num{1.26e-02} & $1.0$ \\
      $10^{-4}$ & \num{2.41e01} & \num{1.95e00} & \num{1.43e-01} & \num{2.64e-02} & \num{1.26e-02} & $1.1$ \\
      $10^{-8}$ & \num{2.56e01} & \num{2.43e00} & \num{2.80e-01} & \num{4.17e-02} & \num{1.33e-02} & $1.7$ \\
      $10^{-12}$ & \num{2.56e01} & \num{2.43e00} & \num{2.80e-01} & \num{4.17e-02} & \num{1.33e-02} & $1.7$ \\
      	\midrule
	\multicolumn{7}{c}{Flux}\\
	\midrule
      $10^{0}$ & \num{6.86e-01} & \num{1.41e-01} & \num{6.39e-02} & \num{3.18e-02} & \num{1.59e-02} & $1.0$ \\
      $10^{-4}$ &  \num{1.03e02} & \num{1.98e01} & \num{3.91e00} & \num{5.90e-01} & \num{7.87e-02} & $2.9$ \\
      $10^{-8}$ & \num{1.07e02} & \num{2.35e01} & \num{6.56e00} & \num{1.75e00} & \num{4.46e-01} & $2.0$ \\
      $10^{-12}$ & \num{1.07e02} & \num{2.35e01} & \num{6.56e00} & \num{1.75e00} & \num{4.47e-01} & $2.0$ \\
      \bottomrule
    \end{tabular}
  \end{center}
\end{table}

\begin{table}
\caption{Fixed hydraulic conductivity $\kappa = 1$, varying storage $0
  < c_0 \leq 1$ for the (minimally) Stokes-Biot stable pairing $\CG^2_2(\triang_h) \times RT_0(\triang_h) \times
  \DG_0(\triang_h)$. The format follows that of Table~\ref{tab:numerics:p2rt0p0:A} 
  and the relative displacement errors are identical.  Compare with Table~\ref{tab:numerics:p2p1p0:C}}
  \label{tab:numerics:p2rt0p0:C}
  \begin{center}
    \begin{tabular}{l|lllll|c}
      \toprule
      \diagbox[]{$c_0$}{$h$} & 1/8 & 1/16 & 1/32 & 1/64 & 1/128 & Rate \\
	\midrule
	\multicolumn{7}{c}{Displacement, c.f.~Table \ref{tab:numerics:p2rt0p0:A}}\\
      	\midrule
	\multicolumn{7}{c}{Pressure}\\
	\midrule
      $10^{0}$  & \num{2.61e-01} & \num{1.02e-01} & \num{5.05e-02} & \num{2.53e-02} & \num{1.26e-02} & $1.0$ \\
      $10^{-12}$ & \num{2.45e-01} & \num{1.01e-01} & \num{5.05e-02} & \num{2.53e-02} & \num{1.26e-02} & $1.0$ \\
      	\midrule
	\multicolumn{7}{c}{Flux}\\
	\midrule
      $10^{0}$ & \num{6.86e-01} & \num{1.41e-01} & \num{6.39e-02} & \num{3.18e-02} & \num{1.59e-02} & $1.0$ \\
      $10^{-12}$ & \num{6.88e-01} & \num{1.41e-01} & \num{6.39e-02} & \num{3.18e-02} & \num{1.59e-02} & $1.0$ \\
      \bottomrule
    \end{tabular}
  \end{center}
\end{table}

\subsection{Convergence of a minimally Stokes-Biot stable pairing}
\label{sec:numerics:Stokes-Biot-Minimal}
We now turn to consider the convergence properties for the pairing $U_h
\times W_h \times Q_h = \CG^2_2(\triang_h) \times \CG^2_1(\triang_h)
\times \DG_0(\triang_h)$ and again report on the relative
approximation errors for the displacement, pressure and flux. This
pairing does not satisfy a Darcy stability condition, for any value of $\kappa$,  
as advanced in the original Stokes-Biot stability criteria; it does satisfy 
the minimally Stokes-Biot criterion of Definition~\ref{defn:minimal-stokes-biot-stability}.  
%
Numerical results for this minimally Stokes-Biot stable discretization, 
for the three paradigms considered in Section~\ref{sec:numerics:Stokes-Biot-Standard}, 
are presented in \ref{sec:numerics:Stokes-Biot-Minimal:a}-%
\ref{sec:numerics:Stokes-Biot-Minimal:c} alongside specific comparisons to the 
standard Stokes-Biot stable case.  

The results of this comparison supply 
computational evidence that Definition~\ref{defn:original-stokes-biot-stability}(iii) can be
replaced by Definition~\ref{defn:minimal-stokes-biot-stability}(iii)
while retaining the convergence properties first observed in
\cite{hong2017parameter,rodrigo2018new}.  Since the Darcy stability of 
Definition~\ref{defn:original-stokes-biot-stability}(iii) is not satisfied \cite{baerland2018uniform} 
uniformly in $\kappa$, our observations strongly suggest that the minimal Stokes-Biot 
stability assumptions, specifically Definition~\ref{defn:minimal-stokes-biot-stability}(iii), 
are in fact, the key component for discretizations that retain their convergence 
properties as $\kappa$ tends to zero.

  \subsubsection{Vanishing storage $c_0 = 0$, varying conductivity $0 < \kappa \leq 1$} 
  \label{sec:numerics:Stokes-Biot-Minimal:a}
  (see: Table~\ref{tab:numerics:p2p1p0:A}) Comparing
  Table~\ref{tab:numerics:p2p1p0:A} with
  Table~\ref{tab:numerics:p2rt0p0:A}, we observe that the performance
  of the two element pairings is almost surprisingly similar. Again,
  the displacement converges at the optimal and expected rate (2), the
  pressure and flux errors increase with decreasing $\kappa$, but
  stabilize, and converge with decreasing mesh size. We further
  observe that the relative errors for the flux for this element
  pairing is smaller than for the $\CG^2_2 \times RT_0 \times \DG_0$
  case (bottom rows).  For a comparison to a discretization satisfying Darcy 
  stability (though not uniformly in $\kappa$) see Section 
  \ref{sec:numerics:Stokes-Biot-Standard:a} and Table~\ref{tab:numerics:p2rt0p0:A}.

  \subsubsection{Fixed storage $c_0 = 1$, varying conductivity $0 < \kappa \leq 1$} 
  \label{sec:numerics:Stokes-Biot-Minimal:b}
  (see: Table~\ref{tab:numerics:p2p1p0:B}) Comparing
  Table~\ref{tab:numerics:p2p1p0:B} with
  Table~\ref{tab:numerics:p2rt0p0:B}, we again observe highly
  comparable performance. The observations made for the $\CG^2_2 \times
  RT_0 \times \DG_0$ case thus also apply for $\CG^2_2 \times \CG^2_1
  \times \DG_0$.  For comparison, see Section 
  \ref{sec:numerics:Stokes-Biot-Standard:b} and Table~\ref{tab:numerics:p2rt0p0:B}.

  \subsubsection{Fixed conductivity $\kappa = 1$, varying storage $0 \leq c_0 \leq 1$} 
  \label{sec:numerics:Stokes-Biot-Minimal:c}
  (see: Table~\ref{tab:numerics:p2p1p0:C}) For this case, we
  observe similar convergence rates as e.g.~in
  Table~\ref{tab:numerics:p2rt0p0:C}). The pressure error increases
  very moderately with decreasing $c_0$ (it doubles as $c_0$ is
  reduced by 12 orders of magnitude), but both the pressure and flux
  converges at the optimal and expected rate (1). For comparison, see Section 
  \ref{sec:numerics:Stokes-Biot-Standard:c} and Table~\ref{tab:numerics:p2rt0p0:C}.

\begin{table}[h]
\caption{Vanishing storage coefficient $c_0 = 0$, varying conductivity
  $0 < \kappa \leq 1$ for the minimally Stokes-Biot stable pairing
  $\CG^2_2(\triang_h) \times \CG^2_1(\triang_h) \times \DG_0(\triang_h)$.
  Listed are the relative displacement (top), relative flux (middle) and relative 
  pressure (bottom) errors for varying $\kappa$ on a series of uniform meshes 
  $\triang_h$ with mesh size $h$.  The displacement errors for 
  $\kappa = 10^{-4}, 10^{-8}$ were identical to the data presented 
  ($\kappa =1$, $\kappa = 10^{-12}$) and are suppressed.  The last column 
  `Rate' denotes the order of convergence using for the last two values in 
  each row.  Compare with Table~\ref{tab:numerics:p2rt0p0:A} }
  \label{tab:numerics:p2p1p0:A}
  \begin{center}
    \begin{tabular}{l|lllll|c}
      \toprule
      \diagbox[]{$\kappa$}{$h$} & 1/8 & 1/16 & 1/32 & 1/64 & 1/128 & Rate \\
      	\midrule
	\multicolumn{7}{c}{Displacement}\\
	\midrule
      $10^{0}$  &  \num{1.64e-01} & \num{4.45e-02} & \num{1.13e-02} & \num{2.84e-03} & \num{7.14e-04} & $2.0$ \\
      $10^{-12}$ & \num{1.64e-01} & \num{4.45e-02} & \num{1.13e-02} & \num{2.84e-03} & \num{7.11e-04} & $2.0$ \\
      	\midrule
	\multicolumn{7}{c}{Pressure}\\
	\midrule
      $10^{0}$  & \num{7.83e01} & \num{1.35e01} & \num{5.30e00} & \num{2.64e00} & \num{1.34e00} & $1.0$ \\
      $10^{-4}$ &  \num{1.22e02} & \num{1.16e01} & \num{1.21e00} & \num{1.39e-01} & \num{2.07e-02} & $2.7$ \\
      $10^{-8}$ & \num{1.25e02} & \num{1.21e01} & \num{1.26e00} & \num{1.43e-01} & \num{2.09e-02} & $2.8$ \\
      $10^{-12}$ & \num{1.25e02} & \num{1.21e01} & \num{1.26e00} & \num{1.43e-01} & \num{2.09e-02} & $2.8$ \\
      	\midrule
	\multicolumn{7}{c}{Flux}\\
	\midrule
      $10^{0}$ &  \num{6.11e-01} & \num{1.51e-01} & \num{7.23e-02} & \num{3.62e-02} & \num{1.81e-02} & $1.0$ \\
      $10^{-4}$ & \num{1.39e02} & \num{1.52e01} & \num{1.54e00} & \num{1.27e-01} & \num{9.42e-03} & $3.7$ \\
      $10^{-8}$ & \num{1.45e02} & \num{1.76e01} & \num{2.14e00} & \num{2.43e-01} & \num{2.87e-02} & $3.1$ \\
      $10^{-12}$ &  \num{1.45e02} & \num{1.76e01} & \num{2.14e00} & \num{2.43e-01} & \num{2.87e-02} & $3.1$ \\
      \bottomrule
    \end{tabular}
  \end{center}
\end{table}

\begin{table}
\caption{Fixed storage coefficient $c_0 = 1$, varying conductivity $0
  < \kappa \leq 1$ for the minimally Stokes-Biot stable pairing 
  $\CG^2_1(\triang_h) \times \DG_0(\triang_h)$.  The format follows that of 
  Table~\ref{tab:numerics:p2p1p0:A} and the relative displacement errors are 
  identical. Compare with Table~\ref{tab:numerics:p2rt0p0:B}.}
  \label{tab:numerics:p2p1p0:B}
  \begin{center}
    \begin{tabular}{l|lllll|c}
      \toprule
      \diagbox[]{$\kappa$}{$h$} & 1/8 & 1/16 & 1/32 & 1/64 & 1/128 & Rate \\
	\midrule
	\multicolumn{7}{c}{Displacement, c.f.~Table \ref{tab:numerics:p2p1p0:A}}\\
	\midrule
	\multicolumn{7}{c}{Pressure}\\
      	\midrule
      $10^{0}$   &  \num{1.55e01} & \num{3.15e00} & \num{1.33e00} & \num{6.67e-01} & \num{3.36e-01} & $1.0$ \\
      $10^{-4}$  & \num{2.54e01} & \num{2.41e00} & \num{2.78e-01} & \num{4.16e-02} & \num{1.33e-02} & $1.6$ \\
      $10^{-8}$  & \num{2.56e01} & \num{2.43e00} & \num{2.80e-01} & \num{4.17e-02} & \num{1.33e-02} & $1.7$ \\
      $10^{-12}$ & \num{2.56e01} & \num{2.43e00} & \num{2.80e-01} & \num{4.17e-02} & \num{1.33e-02} & $1.7$ \\
      	\midrule
	\multicolumn{7}{c}{Flux}\\
	\midrule
      $10^{0}$   & \num{5.97e-01} & \num{1.49e-01} & \num{7.21e-02} & \num{3.61e-02} & \num{1.81e-02} & $1.0$ \\
      $10^{-4}$  & \num{3.14e01} & \num{3.42e00} & \num{3.69e-01} & \num{3.40e-02} & \num{3.26e-03} & $3.4$ \\
      $10^{-8}$  & \num{3.16e01} & \num{3.49e00} & \num{3.90e-01} & \num{3.98e-02} & \num{4.33e-03} & $3.2$ \\
      $10^{-12}$ & \num{3.16e01} & \num{3.49e00} & \num{3.90e-01} & \num{3.98e-02} & \num{4.33e-03} & $3.2$ \\
      \bottomrule
    \end{tabular}
  \end{center}
\end{table}
\begin{table}
  \caption{Fixed hydraulic conductivity $\kappa = 1$, varying storage
    $0 < c_0 \leq 1$ for the minimally Stokes-Biot stable pairing 
    $\CG^2_2(\triang_h) \times \CG^2_1(\triang_h) \times \DG_0(\triang_h)$.  
    Compare with Table~\ref{tab:numerics:p2rt0p0:C}.}
  \label{tab:numerics:p2p1p0:C}
  \begin{center}
    \begin{tabular}{l|lllll|c}
      \toprule
      \diagbox[]{$c_0$}{$h$} & 1/8 & 1/16 & 1/32 & 1/64 & 1/128 & Rate \\
	\midrule
	\multicolumn{7}{c}{Displacement, c.f.~Table \ref{tab:numerics:p2p1p0:A}}\\
      	\midrule
	\multicolumn{7}{c}{Pressure}\\
	\midrule
      $10^{0}$  &  \num{1.55e01} & \num{3.15e00} & \num{1.33e00} & \num{6.67e-01} & \num{3.36e-01} & $1.0$ \\
      $10^{-12}$ & \num{7.83e01} & \num{1.35e01} & \num{5.30e00} & \num{2.64e00} & \num{1.34e00} & $1.0$ \\ 
	\midrule
	\multicolumn{7}{c}{Flux}\\
	\midrule
      $10^{0}$ & \num{5.97e-01} & \num{1.49e-01} & \num{7.21e-02} & \num{3.61e-02} & \num{1.81e-02} & $1.0$ \\
      $10^{-12}$ &  \num{6.11e-01} & \num{1.51e-01} & \num{7.23e-02} & \num{3.62e-02} & \num{1.81e-02} & $1.0$ \\
      \bottomrule
    \end{tabular}
  \end{center}
\end{table}

\section{Conclusion}
\label{sec:concluding-remarks}
The important concept of Stokes-Biot stability, introduced independently by 
\cite{hu2017nonconforming,lee2018,lotfian2018,rodrigo2018new}, has proven
a practical key to the selection of conforming Euler-Galerkin
discretizations of Biot's equations \eqref{eq:biot} that retain their convergence 
properties as the hydraulic conductivity ($0<\kappa$) becomes arbitrarily small.  
The novel contributions of this manuscript are primarily theoretical in nature; 
we have shown that the Stokes-Biot stability perspective can be, formally, relaxed 
and we have introduced the notion of a minimally Stokes-Biot stable Euler-Galerkin 
discretization.  The stability of minimally Stokes-Biot stable schemes is 
independent of both $c_0$ and $\kappa$ (c.f.~ Section~\ref{sec:euler-galerkin:well-posed:subsec:proof}, 
\cite[Theorem 1]{hu2017nonconforming} and 
\cite[Theorem 3.2, Case I]{hong2017parameter}), and we have presented a convergence 
analysis in this context.  

In particular, we differ from previous authors \cite{rodrigo2018new} by carrying 
out our convergence analysis without the use of a Galerkin projection based on 
the Darcy problem.  In doing so, we are able to depart from both the Darcy stability 
assumption, in general, and any questions regarding the appropriate norms for 
uniform-in-$\kappa$ Darcy stability.  In fact, an analysis based on a 
uniform-in-$\kappa$ Darcy stability assumption should take into account 
a pressure-space norm exhibiting one of the forms discussed in 
Section~\ref{subsec:stokes-biot-stability-criteria:darcy-stability-condition}; 
namely $L^2 + \kappa^{1/2} H^1$, $\kappa^{1/2}L^2$ or $\kappa^{1/2}H^1$.  Each 
of these pressure norms have related difficulties over the usual pressure $L^2$ norm 
used here.  First, it is not entirely clear to the authors that the 
$L^2+\kappa^{1/2} H^1$ norm can be treated with the otherwise-standard arguments 
presented here and in related \cite{hong2017parameter,hu2017nonconforming,rodrigo2018new,lee2018} work.  %
Second, the $\kappa^{1/2}L^2$ and $\kappa^{1/2}H^1$ weightings both degenerate 
as $\kappa$ becomes small and must be balanced by an appropriate 
displacement norm so that Stokes stability (Definition~\ref{defn:original-stokes-biot-stability}(i)--(ii) 
and Definition~\ref{defn:minimal-stokes-biot-stability}(i)--(ii)) holds uniformly 
in $\kappa$ as well.  Conversely, our arguments bring together many standard techniques and makes definitevely 
clear, by abdicating Darcy stability, that one need not consider any $\kappa$-weighted norms 
for the pressure space in order to ensure stability and approximation as $\kappa$ diminishes.
Moreover, the convergence analysis presented here is the first instance, of which we are aware, 
of an analysis carried out in the context of the full norm used to prove the Banach-Ne\u{c}as-Babu\u{s}ka 
stability (c.f.~ Section~\ref{sec:euler-galerkin:well-posed:subsec:proof}) of the 
Euler-Galerkin discretization  \eqref{eq:biot:full-var-disc}.  Thus, as neither 
the current convergence analysis, nor the previously-established BNB stability 
result, rely on a Darcy stability assumption, Proposition~\ref{cor:conv-est-specialized} 
solidifies, and generalizes, previous convergence estimates \cite{rodrigo2018new}.  
The concept of minimal Stokes-Biot stability therefore broadens the original view 
of Stokes-Biot stability to include alternative spaces that may not be Darcy stable; 
even for a fixed choice of $\kappa$.

\subsection*{Further observations and practical considerations}

The primary contribution of the current work is theoretical in nature; we have, 
in practice, removed the Darcy restriction for Stokes-Biot stability and demonstrated 
an alternative convergence analysis in this context. Nevertheless, 
practical questions regarding the suitability of both Stokes-Biot and minimally 
Stokes-Biot stable approaches, solving \eqref{eq:biot:var}, can be asked.  In 
particular, we now briefly discuss: the drawbacks of Stokes-Biot and minimally 
Stokes-Biot stable discretizations; what computational advantages, if any, are 
granted by the minimal Stokes-Biot perspective; and alternatives to the boundary 
conditions \eqref{eq:bcs}.

Both Stokes-Biot and minimally Stokes-Biot stables are not without their drawbacks.  
The requirements of both definition~\ref{defn:original-stokes-biot-stability} 
and definition~\ref{defn:minimal-stokes-biot-stability} are general; however, in 
practice, both approaches typically make use of discontinuous pressures. This 
theme is present in the literature for both conformal and non-conformal discretizations.  
In practice, the need of a discontinuous pressure space imposes restrictions on 
the choice of elements.  In this manuscript we have used 
$P_2^d \times RT_0 \times DG_0$ and $P_2^d \times P_1^d \times DG_0$ 
discretizations to illustrate a simple comparison (c.f.~Section \ref{sec:numerics}) 
via numerical experiments in 2D.  In two dimensions, as we discussed in 
Section~\ref{sec:euler-galerkin:well-posed}, one could also consider pairings 
of Scott-Vogelius type, i.e. $P_k^d \times RT_m \times DG_{k-1}$ or 
$P_k^d \times P_m^d \times DG_{k-1}$ where $k \geq 4$.  In the context of minimally 
Stokes-Biot stable triples, we have that the flux space degree can be chosen as  
$0<m \leq k-1$ in both cases; from the original Stokes-Biot point of view, one would 
require $m=k-1$ for the case or Raviart-Thomas elements and polynomial fluxes would 
not be admissible at all.  In 3D, one could also consider extensions of the Scott-Vogelius 
elements \cite{guzman2019}, the enriched cubic displacement element 
and piecewise constant pressures introduced by Guzm\'an and Neilan \cite{guzman2013}, 
the bubble-enriched continuous linear displacement element and piecewise constant pressures 
as in \cite{rodrigo2018new}, or the related bubble-enriched continuous quadratic elements 
with discontinuous linear pressures \cite{GUERMONDERN}; these spaces could be 
considered alongside fluxes of Raviart-Thomas, Brezzi-Douglas-Marini, and Lagrange type.  
One may also consider quadrilateral meshes by using the Stokes pairing \cite{GUERMONDERN}    
given by $Q^2_2 \times DG_1$ along with, for instance, $RT_m$ ($m=0$ or $1$), $BDM_k$ or 
$P_k^d$ ($k=1$ or $2$) fluxes. 

It is practical to note that minimally Stokes-Biot stable discretizations do 
not necessarily confer a computational advantage over those with Darcy stable (for 
fixed $\kappa$) flux-pressuring pairings when equal-order fluxes are selected.  
That is to say, for instance, that the $RT_k$ fluxes will have fewer DOFs than 
the alternative $P^d_k$ fluxes discussed in this manuscript.  However, minimal 
Stokes-Biot stability makes it clear that one can lower the order of the 
flux space without adversely impacting the stability and convergence of the method. 
One could interpret this as a form of `computational advantage' of minimal 
Stokes-Biot.  Overall, however, this is not the important point of minimal 
Stokes-Biot stability.  The important points are that: Darcy stability is not necessary; 
typical `Darcy stable pairings' satisfy the minimal Stokes-Biot 
criteria; and that approximation in both contexts yield strikingly similar results.  
Indeed, numerical experiments (Section \ref{sec:numerics}) show 
similar errors both with (Table~\ref{tab:numerics:p2rt0p0:A}--Table~\ref{tab:numerics:p2rt0p0:C}) 
and without (Table~\ref{tab:numerics:p2p1p0:A}--Table~\ref{tab:numerics:p2p1p0:C}) 
a Darcy stability assumption; even as $\kappa$ becomes very small.  Moreover, we would not expect an improvement in 
results if the norms of the flux and pressure were altered to provide for 
uniform-in-$\kappa$ Darcy stability condition (c.f.~(A) and (B) of %
Section~\ref{subsec:stokes-biot-stability-criteria:darcy-stability-condition})); 
this is due to the fact that approximation of the pressure in the  
$L^2 + \kappa^{1/2} H^1$ norm is similar to that of the $L^2$ norm while approximation 
in the other option, $\kappa^{1/2} L^2$, degrades as $\kappa$ becomes small.  Thus, 
the tenets of minimal Stokes-Biot stability (Definition~\ref{defn:minimal-stokes-biot-stability})
provide an approximation of the pressure in the most sensible norm; that is, the 
$L^2$ norm is a fortuitous choice for convergence analysis, assures the proper 
context for Stokes stability, and does not degrade as $\kappa$ becomes small.  Our 
conclusion is that one can think, instead, in terms of 
Definition~\ref{defn:minimal-stokes-biot-stability} (iii) when designing, or analyzing, 
approaches for Biot when $\kappa \rightarrow 0$.  This important point could certainly 
impact the design, or choice, of discretizations that do in fact confer a computational 
advantage of those where Darcy stability is a requirement.  

Finally, we close with a brief revisitation of the boundary conditions discussed 
in Section~\ref{sec:bcs}.  Extending \eqref{eq:bcs} to inhomogeneous data is 
not a concern; essential boundary conditions conditions can be lifted by selecting 
a particular solution and natural boundary conditions yield right-hand side terms 
that vanish in the error equations \eqref{eq:a-priori:full-model:err-eq}, and do 
not alter the stability arguments (Section~\ref{sec:euler-galerkin:well-posed}). 
However, it is valid to note that the assumption that $\Gamma_f = \Gamma_c$ and 
$\Gamma_t = \Gamma_p$, in \eqref{eq:bcs:general}, may not be practical for 
problems of interest.  The conditions \eqref{eq:bcs} were considered in  
the original Stokes-Biot, or motivating, literature \cite{hu2017nonconforming,rodrigo2018new} 
which lead to their adoption here.  The advantage of the boundary conditions \eqref{eq:bcs} is 
that they provide for an overall discussion that scopes naturally between the case where 
$\Gamma_c \neq \partial \Omega$ and $\Gamma_c = \partial\Omega$ provided that 
$|\Gamma_c| > 0$ is assumed.  In particular, as mentioned in 
Remark~\ref{rmk:clampedbd}, if $\Gamma_c = \partial\Omega$ then conditions 
\eqref{eq:bcs} imply $\Gamma_t = \emptyset$, the variational forms  \eqref{eq:biot:var} 
and \eqref{eq:biot:full-var-disc} are unchanged, and all results discussed hold 
when $Q = L_0^2(\Omega)$ is selected, instead, in \eqref{eq:biot:spaces}.   

It is reasonable to ask what other boundary condition configurations can be considered, 
and under what conditions.  First, we note that the requirement that 
$\Gamma_c \cap \Gamma_t = \emptyset$ and $\Gamma_f \cap \Gamma_p = \emptyset$ arise 
from the early work in well posedness for Biot \cite{showalter-2000}.  Moreover, 
the requirement that the positive measure of the clamped displacement boundary is 
non-zero, i.e.~that $|\Gamma_c| > 0$, provides the coercive property 
$a(u,u) \geq \gamma_a \norw{u}{1}^2$ needed by both 
Definition~\ref{defn:original-stokes-biot-stability} (Stokes-Biot stability) and 
Definition~\ref{defn:minimal-stokes-biot-stability} (minimal Stokes-Biot stability); 
this is therefore a strict requirement of the proposed method.  However, 
if both $|\Gamma_c| > 0$ and $|\Gamma_t| > 0$ then, as noted in \cite{hong2020parameter} 
and used in \cite{lee2018}, the conditions \eqref{eq:bcs} can be relaxed to those of \eqref{eq:bcs:general}. In 
this case, the requirements of both Definition~\ref{defn:original-stokes-biot-stability} 
and Definition~\ref{defn:minimal-stokes-biot-stability} can be satisfied with 
$Q = L_2(\Omega)$ in \eqref{eq:biot:spaces}.  In this case, the variational 
formulations \eqref{eq:biot:var} and \eqref{eq:biot:full-var-disc} are, again, 
unaltered and the results of the manuscript follow analagously.  It is true, as 
discussed above, that restrictive boundary conditions, i.e.~such as \eqref{eq:bcs}, 
are needed when $\Gamma_c = \partial\Omega$ in order to ensure the tenets of 
both Definition~\ref{defn:original-stokes-biot-stability} and 
Definition~\ref{defn:minimal-stokes-biot-stability}; this is not an additional 
imposition of minimal Stokes-Biot stability (Definition~\ref{defn:minimal-stokes-biot-stability})
but rather of the Stokes-Biot perspective in general.

\section{Acknowledgements}
M.~E.~Rognes has received funding from the European Research Council (ERC) under 
the European Union's Horizon 2020 research and innovation programme under grant 
agreement 714892.  The research works of M.~E.~Rognes and T.~B.~Thompson were 
also supported by the Research Council of Norway under the FRINATEK Young 
Research Talents Programme; project number 250731/F20 (Waterscape). 
The work of K.-A.~Mardal was supported by the Research Council of Norway grant 
number 301013.  

\bibliographystyle{spmpsci}
\bibliography{references}

\begin{thebibliography}{10}
\providecommand{\url}[1]{{#1}}
\providecommand{\urlprefix}{URL }
\expandafter\ifx\csname urlstyle\endcsname\relax
  \providecommand{\doi}[1]{DOI~\discretionary{}{}{}#1}\else
  \providecommand{\doi}{DOI~\discretionary{}{}{}\begingroup
  \urlstyle{rm}\Url}\fi

\bibitem{fenics-one}
Aln{\ae}s, M., Blechta, J., Hake, J., Johansson, A., Kehlet, B., Logg, A.,
  Richardson, C., Ring, J., Rognes, M., Wells, G.: The {FE}ni{CS} {P}roject
  {V}ersion 1.5.
\newblock Archive of Num. Soft. \textbf{3} (2015)

\bibitem{baerland2018uniform}
B{\ae}rland, T., Kuchta, M., Mardal, K.A., Thompson, T.: {A}n {O}bservation
  {O}n {T}he {U}niform {P}reconditioners {F}or {T}he {M}ixed {D}arcy {P}roblem.
\newblock Numer. Methods Partial Differential Equations \textbf{36}(6),
  1718--1734 (2020).
\newblock \doi{10.1002/num.22500}

\bibitem{lofstrom1976}
Bergh, J., L\"{o}fstr\"{o}m, J.: {I}nterpolation {S}paces.
\newblock A series of comprehensive studies in mathematics. Springer, New York
  (1976)

\bibitem{boffi-brezzi-fortin2013}
Boffi, D., Brezzi, F., Fortin, M.: {M}ixed {F}inite {E}lement {M}ethods and
  {A}pplications, 1 edn.
\newblock Springer-Verlag (2013)

\bibitem{braess2002finite}
Braess, D.: Finite elements: Theory, fast solvers, and applications in solid
  mechanics (2nd edition).
\newblock Cambridge University Press (2002)

\bibitem{brezzi1974existence}
Brezzi, F.: On the existence, uniqueness and approximation of saddle-point
  problems arising from {L}agrangian multipliers.
\newblock Publications math{\'e}matiques et informatique de Rennes \textbf{S4},
  1--26 (1974)

\bibitem{brun2019monolithic}
Brun, M.K., Ahmed, E., Berre, I., Nordbotten, J.M., Radu, F.A.: Monolithic and
  splitting based solution schemes for fully coupled quasi-static
  thermo-poroelasticity with nonlinear convective transport.
\newblock arXiv preprint arXiv:1902.05783  (2019)

\bibitem{GUERMONDERN}
Ern, A., Guermond, J.L.: {T}heory and {P}ractice of {F}inite {E}lements.
\newblock Springer (2004)

\bibitem{ern-munier-2009}
Ern, A., Meunier, S.: A posteriori error analysis of euler-galerkin
  approximations to coupled elliptic-parabolic problems.
\newblock ESAIM: M2AN \textbf{43}(2), 353--375 (2009).
\newblock \doi{10.1051/m2an:2008048}.
\newblock \urlprefix\url{https://doi.org/10.1051/m2an:2008048}

\bibitem{evans10}
Evans, L.: Partial differential equations.
\newblock American Mathematical Society, Providence, R.I. (2010)

\bibitem{girault2019priori}
Girault, V., Wheeler, M.F., Almani, T., Dana, S.: A priori error estimates for
  a discretized poro-elastic--elastic system solved by a fixed-stress
  algorithm.
\newblock Oil \& Gas Science and Technology--Revue d’IFP Energies nouvelles
  \textbf{74}, 24 (2019)

\bibitem{vardakis2019}
Guo, L., Li, Z., Ventikos, Y.e.a.: {O}n the {V}alidation of a
  {M}ultiple-{N}etwork {P}oroelastic {M}odel {U}sing {A}rterial {S}pin
  {L}abeling {MRI} {D}ata.
\newblock Front. Comput. Neurosci. \textbf{13}, 60 (2019)

\bibitem{vardakis2018}
Guo, L., Vardakis, J., Ventikos, Y.e.a.: {S}ubject-specific multi-poroelastic
  model for exploring the risk factors associated with the early stages of
  {A}lzheimer's disease.
\newblock Interface Focus \textbf{8}(1), 20170,019 (2018)

\bibitem{guzman2013}
Guzman, J., Neilan, M.: Conforming and divergence-free stokes elements in three
  dimensions.
\newblock IMA J. Numer. Anal. \textbf{34}(4), 1489--1508 (2019).
\newblock \doi{10.1090/mcom/3346}

\bibitem{guzman2019}
Guzman, J., Scott, L.: The scott-vogelius finite elements revisited.
\newblock Math. Comp. \textbf{88}, 515--529 (2019).
\newblock \doi{10.1090/mcom/3346}

\bibitem{herrmann1965elasticity}
Herrmann, L.R.: Elasticity equations for incompressible and nearly
  incompressible materials by a variational theorem.
\newblock AIAA journal \textbf{3}(10), 1896--1900 (1965)

\bibitem{hong2017parameter}
Hong, Q., Kraus, J.: Parameter-robust stability of classical three-field
  formulation of {B}iot's consolidation model.
\newblock Electron.~T.~Numer.~Ana. \textbf{48}, 202--226 (2018)

\bibitem{hong2020parameter}
Hong, Q., Kraus, J., Lymbery, M., Wheeler, M.F.: Parameter-robust convergence
  analysis of fixed-stress split iterative method for multiple-permeability
  poroelasticity systems.
\newblock Multiscale Modeling \& Simulation \textbf{18}(2), 916--941 (2020)

\bibitem{hu2017nonconforming}
Hu, X., Rodrigo, C., Gaspar, F.J., Zikatanov, L.: A nonconforming finite
  element method for the {B}iot's consolidation model in poroelasticity.
\newblock Journal of Computational and Applied Mathematics \textbf{310},
  143--154 (2017)

\bibitem{kraus2020parameter}
Kraus, J., Lederer, P., Lymbery, M., Schoberl, J.: Uniformly well-posed
  hybridized discontinuous {G}alerkin/hybrid mixed discretizations for {B}iot's
  consolidation model.
\newblock Cold Spring Harbor Lab. (preprint) arXiv:2012.08584  (2020)

\bibitem{kumar2020conservative}
Kumar, S., Oyarz{\'u}a, R., Ruiz-Baier, R., Sandilya, R.: Conservative
  discontinuous finite volume and mixed schemes for a new four-field
  formulation in poroelasticity.
\newblock ESAIM: Mathematical Modelling and Numerical Analysis \textbf{54}(1),
  273--299 (2020)

\bibitem{lee2018}
Lee, J.: Robust three-field finite element methods for {B}iot's consolidation
  model in poroelasticity.
\newblock BIT Numer. Math. \textbf{58}(2), 347--372 (2018)

\bibitem{lee2017parameter}
Lee, J., Mardal, K.A., Winther, R.: Parameter-robust discretization and
  preconditioning of {B}iot's consolidation model.
\newblock SIAM Journal on Scientific Computing \textbf{39}(1), A1--A24 (2017)

\bibitem{lee2019}
Lee, J., Piersanti, E., Mardal, K.A., Rognes, M.: A mixed finite element method
  for nearly incompressible multiple-network poroelasticity.
\newblock SIAM Journal of Scientific Computing \textbf{41}(2), A722--A747
  (2019)

\bibitem{li2013}
Li, X., Holst, H., Kleiven, S.: {I}nfluences of brain tissue poroelastic
  constants on intracranial pressure {(ICP)} during constant-rate infusion.
\newblock Comput. Methods Biomech. Biomed. Eng. \textbf{16}(12), 1330--1343
  (2013)

\bibitem{lipnikov-2002}
Lipnikov, K.: Numerical methods for the {B}iot model in poroelasticity.
\newblock Ph.D. thesis, University of Houston (2002)

\bibitem{lotfian2018}
Lotfian, Z., Sivaselvan, M.: Mixed finite element formulation for dynamics of
  porous media.
\newblock Int. J. Numer. Methods. Eng. \textbf{115}, 141--171 (2018)

\bibitem{oyarzua2016locking}
Oyarz{\'u}a, R., Ruiz-Baier, R.: Locking-free finite element methods for
  poroelasticity.
\newblock SIAM Journal on Numerical Analysis \textbf{54}(5), 2951--2973 (2016)

\bibitem{riviereDG}
Riviere, B.: {D}iscontinuous {G}alerkin methods for {S}olving {E}lliptic and
  {P}arabolic {E}quations: {T}heory and {I}mplementation.
\newblock Society for Industrial and Applied Mathematics (2008)

\bibitem{rodrigo2018new}
Rodrigo, C., Hu, X., Ohm, P., Adler, J.H., Gaspar, F.J., Zikatanov, L.: New
  stabilized discretizations for poroelasticity and the {S}tokes’ equations.
\newblock Computer Methods in Applied Mechanics and Engineering \textbf{341},
  467--484 (2018)

\bibitem{showalter-2000}
Showalter, R.: {D}iffusion in {P}oro-{E}lastic {M}edia.
\newblock J. of Math. Analysis and App. \textbf{24}(251), 310--340 (2000)

\bibitem{storvik2019optimization}
Storvik, E., Both, J.W., Kumar, K., Nordbotten, J.M., Radu, F.A.: On the
  optimization of the fixed-stress splitting for biot's equations.
\newblock International Journal for Numerical Methods in Engineering
  \textbf{120}(2), 179--194 (2019)

\bibitem{riviere2019}
Thompson, T., Riviere, B., Knepley, M.: An implicit discontinuous galerkin
  method for modeling acute edema and resuscitation in the small intestine.
\newblock Math Med. Biol. \textbf{36}(4), 513--548 (2019)

\bibitem{youngriviere2014}
Young, J., Riviere, B.: A mathematial model of intestinal oedema formation.
\newblock Math Med. Biol. \textbf{31}(1), 1--15 (2014)

\bibitem{zenisek1984}
Zenisek, A.: The existence and uniqueness theorem in {B}iot's consolidation
  theory.
\newblock Aplikace matematiky \textbf{29}(3), 194--211 (1984).
\newblock \urlprefix\url{http://eudml.org/doc/15348}

\end{thebibliography}
\end{document}